\documentclass[11pt]{article}
\usepackage{latexsym}
\usepackage{verbatim}
\usepackage{amsmath}
\usepackage{amsthm}
\usepackage{amssymb}
\usepackage{stmaryrd}
\usepackage{enumerate}
\usepackage{graphicx}
\usepackage[all]{xy}
\def\old#1{}

\newif\ifshortversion
\shortversiontrue

\newcommand{\multiset}[1]{\overline{ \{ } {#1} \overline{\} }}
\newcommand{\bmultiset}[1]{\overline{ \big\{ } {#1} \overline{\big\} }}
\newcommand{\Bmultiset}[1]{\overline{ \bigg\{ } {#1} \overline{\bigg\} }}

\newcommand{\fracpart}[1]{\left\langle #1\right\rangle}
\newtheorem{theorem}{Theorem}[section]
\newtheorem{corollary}[theorem]{Corollary}
\newtheorem{lemma}[theorem]{Lemma}

\newtheorem{example}[theorem]{Example}
\newtheorem{proposition}[theorem]{Proposition}
\newtheorem{construction}[theorem]{Construction}
\newtheorem{problem}[theorem]{Problem}

\def\case{\begin{array}{c@{\quad \mbox{if} \quad} l }}
\def\BB{\mathcal{B}}
\def\CC{\mathcal{C}}
\def\FF{\mathcal{F}}
\def\EE{\mathcal{E}}
\def\HH{\mathcal{H}}
\def\GG{\mathcal{G}}
\def\LL{\mathfrak{L}}

\def\TT{\mathcal{T}}
\def\KK{\mathcal{K}}
\def\AA{\mathcal{A}}
\def\AAA{\mathbb{A}}
\def\MMM{\mathbb{M}}
\def\DD{\mathcal{D}}
\def\Bb{\mathbb{B}}
\def\Pp{\mathbb{P}}
\def\lcm{\hbox{\rm{lcm}}}
\def\upshad{\reflectbox{$\partial $}}
\def\supp{\operatorname{supp}}
\def\high{\operatorname{high}_j}
\def\low{\operatorname{low}_j}
\newcommand{\prtime}{{\count0=\time\divide\count0 by 60
     \count1=-\count0\multiply\count1 by 60 \advance\count1 by \time
     \the\count0:\the\count1} }

\def\myheads#1;#2;{
\pagestyle{myheadings} \markboth{{\sc\hfill
     #1\hfill\protect\makebox[0cm][r]{\rm\today; \prtime}}}
     {{\sc\protect\makebox[0cm][l]{\rm\today;\ \prtime}\hfill
     #2\hfill}} \thispagestyle{myheadings} }
     \def\old#1{}
\title{Mixed  orthogonal arrays,  $k$-dimensional $M$-part Sperner multi-families, and  full multi-transversals}
\author{
     Harout Aydinian\footnote{This author was supported by DFG-project AH46/7-1 ``General Theory of Information Transfer''.}\\
     {\tt ayd@math.uni-bielefeld.de}\\
     University of Bielefeld, POB 100131, D-33501 Bielefeld, Germany\\
     \'Eva Czabarka\footnote{This author have been supported by a PIRA
     grant of the University of South Carolina and the hospitality of the Rheinische Friedrich-Wilhelms Universit\"at, Bonn.}\\
     {\tt czabarka@math.sc.edu}\\
     University of South Carolina, Columbia, SC 29208, USA\\
     L\'aszl\'o A. Sz\'ekely\footnote{This author was supported in part
     by the NSF DMS contracts No. 0701111, No. 1000475, and by the Alexander
     von Humboldt-Stiftung at the Rheinische Friedrich-Wilhelms Universit\"at, Bonn.}\\
     {\tt szekely@math.sc.edu}\\
     University of South Carolina, Columbia, SC 29208, USA\\ \\ \\
     {\sl This paper is dedicated to the memory of Professor Rudolf Ahlswede.}
          }
     \date{}
\begin{document}
\maketitle


\begin{abstract} Aydinian {\sl et al.}  [J. Combinatorial Theory A
{\bf 118}(2)(2011), 702--725] substituted the usual BLYM inequality for $L$-Sperner
families with a set of  $M$ inequalities for $(m_1,m_2,\ldots,m_M;L_1,L_2,\ldots,L_M)$ type
$M$-part Sperner families and showed that if all inequalities hold with equality, then
the family is homogeneous. Aydinian {\sl et al.}  [Australasian J. Comb. {\bf 48}(2010),
133--141] observed that all  inequalities hold with equality if and only if the transversal
of the Sperner family corresponds to a simple mixed orthogonal array with constraint
$M$, strength $M-1$, using $m_i+1$ symbols in the $i^{\text{th}}$ column. In this paper 
we define
$k$-dimensional $M$-part 
Sperner multi-families with parameters $L_P:~P\in\binom{[M]}{k}$ and
prove $\binom{M}{k}$ BLYM inequalities for them.
We show that if $k<M$ and
all inequalities hold with equality, then these multi-families must be homogeneous with profile matrices
that are 
strength $M-k$ mixed orthogonal arrays. 
For $k=M$, homogeneity is not always true, but some necessary 
conditions are given for certain simple families.
Following the methods of
Aydinian {\sl et al.}  [Australasian J. Comb. {\bf 48}(2010),
133--141], we give new constructions to simple mixed orthogonal arrays with constraint
$M$, strength $M-k$, using $m_i+1$ symbols in the $i^{\text{th}}$ column.
We extend the convex hull method to $k$-dimensional $M$-part Sperner multi-families,
and allow additional conditions providing new results even for simple  1-part Sperner families. 
\end{abstract}

\section{Notations}

We will use $[n]=\{1,2,\ldots,n\}$ and $[n]^{\star}=\{0,1,\ldots,n-1\}$, and
let $\binom{X}{\ell}$ denote the family of all $\ell$-element subsets of the set $X$.

We will talk about multisets, where every element appears with some positive integer multiplicity.
We will use the notation $\multiset{\cdot}$ to emphasize that we talk about a multiset. 
If $\AA$ is a multiset, then the support set $\supp(\AA)$ of $\AA$ is the simple set containing all
elements of $\AA$.
We denote the multiplicity of $x$ in a multiset $\AA$ by $\#[x,\AA]$. Clearly, $x\notin \AA$ iff $\#[x,\AA]=0$. 

For shortness, for multisets $\AA$ and simple sets $\BB$ we will use $\AA\Subset\BB$ to denote
$\supp(\AA)\subseteq\BB$, i.e. the event that all the elements of $\AA$ are elements of $\BB$.
If $P(\cdot)$ is a proposition and $k_C$ are non-negative integers, then
$\multiset{C^{k_C}:P(C)}$ denotes the multiset we obtain by taking all objects $C$ with multiplicity $k_C$
that satisfy $P(\cdot)$. Clearly, if $\AA$ is a multiset, then  $\multiset{C^{\#[C,\AA]}:P(C)}$ 
will only contain elements of $\AA$.

If $\AA_i$ are a multi-families of sets
and $P(\cdot)$ is a Boolean polynomial on $\ell$ sets and $k_{C_1,\ldots,C_{\ell}}$ are non-negative integers, then 
$\multiset{C^{k_{C_1,\ldots,C_k}}:C=P(C_1,\ldots,C_{\ell}}$
denotes the multiset where every $C$ appears with multiplicity $\sum_{(C_1,\ldots,C_{\ell})} k_{C_1,\ldots,C_{\ell}}$ where
the sum is taken over all different $\ell$-tuples $(C_1,\ldots,C_{\ell})$ for which $C=P(C_1,\ldots,C_{\ell})$.

For a multiset $\AA$, the {\sl size} or {\sl cardinality} of $\AA$ is $\vert\AA\vert=\sum_{x\in\AA}\#[x,\AA]$.

We use $\uplus$ to denote disjoint unions of multisets; if $\AA$ and $\BB$ are multisets, then $\AA\uplus\BB$ denotes
the multiset obtained by
$\#[x,\AA\uplus\BB]=\#[x,\AA]+\#[x,\BB]$. Clearly, if $\AA$ and $\BB$ are disjoint (simple) sets, then $\uplus$ is the usual (disjoint)
union.

For multisets $\AA$ and $\BB$, $\AA\cup\BB$ denotes the multiset obtained by $\#[x,\AA\cup\BB]=\max(\#[x,\AA],\#[x,\BB])$.

For multisets $\AA$ and $\BB$, $\AA\cap\BB$ denotes the multiset obtained by $\#[x,\AA\cap\BB]=\min(\#[x,\AA],\#[x,\BB])$.

For multisets $\AA$ and $\BB$, $\AA\setminus\BB$ denotes the multiset obtained by
$\#[x,\AA\setminus\BB]=\max(0,\#[x,\AA]-\#[x,\BB])$.

A $\BB$ multiset of subsets of $X$ is a {\sl multichain} of length $|\BB|$, if the elements of $\BB$ are pairwise comparable
(i.e. the different elements of $\BB$ form a chain in the usual sense, and elements may occur with higher multiplicity then $1$). 

A multiset $\BB$ is called an {\sl antichain} if it is a set forming an antichain. Antichains are always simple sets.

Finally, if $\FF$ is a multiset and $k(F)$ is a real-valued function on $\supp(\FF)$, then we use the notation
$$\sum_{F\in\FF} k(F):=\sum_{F\in\supp(\FF)}k(F)\cdot\#[F,\FF].$$

\section{Definitions: $k$-dimensional multi-transversals and mixed orthogonal arrays}
Let us be given $1\leq n_1,\ldots,n_M$, a $k\in[M]$, and set for the rest of the paper $\pi_M=\prod_{i=1}^{M}[n_i]^{\star}$. 
For each $P\in \binom{[M]}{k}$ let us be given
an integer  $L_{P}$ such that
$1\le L_{P}$. A $\TT\Subset\pi_M$ is called  a
{\sl $k$-dimensional multi-transversal\footnote{This concept is different from the {\sl transversal design} in  \cite{hss} even for the simple transversals.} on $\pi_M$} with these parameters if
for every $P\in\binom{[M]}{k}$, fixing $b_j\in [n_j]^{\star}$ arbitrarily for every $j\in[M]\setminus P$,
we have that
\begin{equation}\label{korlat}
\bigg\vert\Bmultiset{(i_1,\ldots,i_M)^{\#[(i_1,\ldots,i_M),\TT]}: i_j=b_j\hbox{ for all }j\in [M]\setminus P}\bigg\vert\le L_{P}.
\end{equation}
If we want to emphasize that $\TT$ is a set and not a multiset (i.e. every element  of $\TT$ has multiplicity $1$), then
we call it a {\sl $k$-dimensional transversal} or a {\sl $k$-dimensional simple transversal}.

It is easy to see that if $\TT$ is a $k$-dimensional multi-transversal, then we have the inequalities
\begin{equation}\label{korlat2}
\forall  P\in\binom{[M]}{k} \hbox{\ \ \ \ \ \ }\big\vert\TT\big\vert\le L_{P}\prod_{j\notin P} n_j.
\end{equation}
A $k$-dimensional multi-transversal is called {\sl full}, if equality holds for at least one
inequality set by a
$P\in\binom{[M]}{k}$. It is clear from the definitions that
equality in one inequality (i.e. having a full transversal) implies equalities in all inequalities
iff
\begin{equation}\label{konstant}
\frac{1}{L_P}\prod_{j\in P} n_j  \hbox{\ \ does not depend on the choice of  }P.
\end{equation}

The $k$-dimensional multi-transversals above have intimate connection
to mixed orthogonal arrays.
Consider  sets $S_i$ of $n_i$ symbols   $(i=1,\ldots,M)$   and consider an
$N \times M$ matrix $T$,  whose the $i^{\text{th}}$ column draws its elements from the
set $S_i.$ This matrix is  called a {\sl mixed} (or {\sl asymmetrical})
 {\sl orthogonal array} or {\sl MOA} (the notion of {\sl orthogonal array with
variable numbers of symbols} is also used), of {\sl strength} $d$,
 {\sl constraint} $M$ and {\sl index set} $\mathbb{L}$, if for
 any choice of $d$ different columns $ j_1, \dots , j_d$ each sequence
 $(a_{j_1},\ldots,a_{j_d})\in S_{j_1} \times \cdots  \times S_{j_d}$
 appears exactly $\lambda(j_1,\ldots,j_d)\in \mathbb{L}$ times after deleting the other
 $M-d$ columns.
 In the case of equal symbol set sizes (and therefore constant $\lambda$)
we have the classical definition of orthogonal arrays.
 A (mixed) orthogonal array is {\sl simple}, if the matrix $T$ has no repeated
 rows. The following proposition easily follows from the definitions.

\begin{proposition} \label{elemi}
If the parameters $n_1,\ldots,n_M, \{L_P:P\in\binom{[M]}{k}\}$ satisfy
the condition {\rm (\ref{konstant})}, then
any full $k$-dimensional multi-transversal is a MOA with symbol sets
$S_i=[n_i]^{\star}$, of constraint
$M$, strength $M-k$, and index set $\mathbb{L}=\{L_P:P\in\binom{[M]}{k}\}$, with $\lambda({j_1},\ldots,{j_{M-k}})=
L_{[M]\setminus\{j_1,\ldots,j_{M-k}\}}$. Furthermore, if the transversal is simple, then so is the MOA. 

Moreover, if
a MOA $T$ is given with
symbol sets $S_i$, (where $n_i=|S_i|$),
of
constraint $M$,  strength $d$, with an index set $\mathbb{L}$,
then $T$ corresponds
to a full $(M-d)$-dimensional multi-transversal with parameters $n_i$ and $L_P=\lambda([M]\setminus P)$. Furthermore, 
if $T$ is simple, then
so is the corresponding multi-transversal.
\end{proposition}
Orthogonal arrays were introduced by  Rao \cite{rao43,rao47}, the terminology
 was introduced by  Bush \cite{b50,b52}.
 Cheng \cite{cheng} seems to be the first  author to consider MOAs.
 MOAs are widely used in planning experiments.  The standard reference work
 for (mixed) orthogonal arrays is the monograph of Hedayat,  Sloane and
 Stufken \cite{hss}.  Constructions for MOAs usually use finite fields
 and few MOAs of strength $>2$ are known.

An alternative formulation to $k$-dimensional (simple) transversals is the following:
a set of length $M$ codewords from $\pi_M$, such that for
every  $P\in\binom{[M]}{k}$ set of character positions, if the characters are prescribed
in any way for the $i\notin P$ character positions, at most $L_P$ of our codewords
show all the prescribed values. In particular, if $L_P$ is identically
1, then a  $k$-dimensional transversal is a code  of minimum Hamming distance $k+1$  (see \cite{ms}).

Also, $k$-dimensional transversals are {\sl packing arrays} and their complements
are {\sl covering arrays} (for the definitions, see \cite{hss}).

\section{$k$-dimensional $M$-part Sperner multi-families}
Let us be given  an underlying set $X$ of cardinality $n$ (often just $X=[n]$),
and a  fixed partition
$X_1,\ldots, X_M$ of $X$ with $|X_i|=m_i$. Set $n_i=m_i+1$ (this convention will be used throughout the paper from now on).

Assume that  $\CC_{i}$ is a (simple) chain in
the subset lattice of $X_i$, for $i\in P$, where $P \subseteq[M]$. We define the product  of these
chains as
$$\prod\limits_{i\in P} \CC_{i}=\left\{\biguplus_{i\in P} A_{i}: A_{i}\in \CC_{i}\right\}.$$
Let us be given
for every $P\in\binom{[M]}{k}$ 
a positive integer $L_P$. 

We call
a multi-family of subsets of $X$,  $\FF$,  a  {\sl $k$-dimensional $M$-part Sperner multi-family 
with parameters $\{L_P: P\in\binom{[M]}{k}\}$},
 if for all $P\in\binom{[M]}{k}$,
for all (simple) chains $\CC_j$ in $X_j$ $(j\in P)$ and for all fixed sets $D_i\subseteq X_i$
$(i\notin P)$ we have that
$$
\Bigg\vert
\Bmultiset{
F^{\#[F,\FF]}:
\Big(F\cap\biguplus_{j\in P} X_j\Big)\in\prod_{j\in P} \CC_j, \forall i\in[M]\setminus P\,\,\, X_i\cap F=D_i
}
\Bigg\vert\le L_P.
$$
A {\sl $k$-dimensional $M$-part Sperner family} or a {\sl simple $k$-dimensional $M$-part Sperner family} $\FF$ is a Sperner
multi-family where $\#[F,\FF]\in\{0,1\}$.
For simple families, for $k=1$ we get back the concept of $M$-part $(m_1,\ldots,m_M;L_1,\ldots,L_M)$-Sperner families from  \cite{tour}, and restricting further with $M=1$, we get back the concept of
the classical
$L$-Sperner families.

The {\sl profile vector} of a subset  $F$ of $X$ is the $M$-dimensional vector
$$(|F\cap X_1|,\ldots,|F\cap X_j|,\ldots,|F\cap X_M|)\in\pi_M.$$
The {\sl profile matrix} $  \Pp(\FF)  =(p_{i_1,\ldots,i_M})_{(i_1,\ldots,i_M)\in\pi_M}$ of a multi-family $\FF$ of subsets of $X$ is an $M$-dimensional matrix,
whose  entries count with multiplicity the elements of $\FF$ with a given profile vector:
$$p_{i_1,i_2,\ldots,i_M}=
\Bigl|\bmultiset{F^{\#[F, \FF]}:\hbox{\  \ } \forall j \hbox{\ } |F\cap X_j|=i_j}\Bigl|.
$$
A multi-family $\FF$
of subsets of $X$ is called {\sl homogeneous}, if  the profile vector of a set
determines the multiplicity of  the set in $\FF$.  In a homogeneous multi-family
$\FF$, we have that for each profile vector $(i_1,\ldots,i_M)$ there is a non-negative integer
$r_{i_1,\ldots,i_M}$ such that 
$p_{i_1,i_2,\ldots,i_M}=r_{i_1,\ldots,i_M}\prod\limits_{j=1}^M\binom{m_j}{ i_j}$. For simple families, $r_{i_1,\ldots,i_M}\in\{0,1\}$, and
this concept of homogeneity simplifies to the usual concept.

Given a homogeneous    $k$-dimensional $M$-part Sperner multi-family  $\FF$ with parameters
 $\{L_P: P\in\binom{[M]}{k}\}$, we observe that the multiset containing each $(i_1,\ldots,i_M)$ with multiplicity
 $r_{i_1,\ldots,i_M}$ is a $k$-dimensional multi-transversal with these parameters,
and  every $k$-dimensional multi-transversal
comes from a  homogeneous     $k$-dimensional $M$-part Sperner multi-family.
The multi-family is a (simple) family precisely when the corresponding multi-transversal is in fact
a simple transversal. 

\section{New Sperner type results}\label{newSp}
In Sections~\ref{newSp}, \ref{hull} and \ref{newMOA} we do not break the narrative
with lengthy proofs and leave those to Sections~\ref{Sproofs}, \ref{hullproofs} and \ref{MOAproofs}.
We start with the following:
\begin{theorem}{\rm [BLYM inequalities]}\label{ujBLYM}
Given a   $k$-dimensional $M$-part Sperner multi-family  $\FF$ 
with parameters
 $\{L_P: P\in\binom{[M]}{k}\}$, the following inequalities hold:
\begin{equation}\label{allBLYM}
\forall P\in\binom{[M]}{k} \hbox{\  \ }
\sum_{(i_1,\ldots,i_M)\in \pi_M}
\frac{p_{i_1,\ldots,i_M}}{\prod\limits_{j=1}^M \binom{m_j}{i_j}}
\le \frac{L_P}{\prod\limits_{j\in P} n_j}\prod_{j=1}^M n_j.
\end{equation}
\end{theorem}
For simple families, the special case of this theorem for $k=1$ was found by
Aydinian,  Czabarka, P. L. Erd\H os,  and  Sz\'ekely
in \cite{tour}, Theorem 6.1.
The special
case for $M=1$ was first in print in \cite{all-2fur}, and the special case $L=M=1$
is the Bollob\'as--Lubell--Meshalkin--Yamamoto (BLYM) inequality
\cite{bollobas, lubbell, meshalkin, yamamoto}. Note that the single classical BLYM
inequality has been substituted by a {\sl family} of inequalities. Cases of equality can be
characterized as follows:
\begin{theorem}\label{ujBLYM=}
Given integers $1= k\leq M$ or $2\leq k\leq M-1$, let  $\FF$ be a $k$-dimensional $M$-part Sperner multi-family 
 with parameters
 $\{L_P: P\in\binom{[M]}{k}\}$  satisfying all inequalities in {\rm (\ref{allBLYM})} with equality. Then the following are true:
 \begin{enumerate}[{\rm (i)}]
 \item $\FF$ is homogeneous;
 \item  $\frac{L_P}{\prod_{j\in P} n_j}$ does not depend on the choice of $P$;
  \item the $k$-dimensional multi-transversal corresponding to $\FF$ is a MOA with symbol sets
$S_i=[n_i]^{\star}$, of constraint
$M$, strength $M-k$, and index set $\mathbb{L}=\{L_P:P\in\binom{[M]}{k}\}$,
 with $\lambda({j_1},\ldots,{j_{M-k}})=L_{[M]-\{j_1,\ldots,j_{M-k}\}}$.
\end{enumerate}
Any MOA, as described in {\rm (iii)}
 is a $k$-dimensional
 multi-transversal on $\pi_M$ with parameters $\{L_P:P\in\binom{[M]}{k}\}$, and it corresponds to the profile
 matrix of a homogeneous $k$-dimensional $M$-part Sperner multi-family  $\FF$ with parameters
 $\{L_P: P\in\binom{[M]}{k}\}$   on a partitioned  $(m_1+\ldots+m_M)$-element underlying set,
  which  satisfies all inequalities in {\rm (\ref{allBLYM})} with equality.
  
Under this correspondence, simple $k$-dimensonal $M$-part Sperner families correspond to simple MOAs.
\end{theorem}
Note that the last sentence is obvious and part (ii) follows directly from the conditions of the theorem.

The special case of this theorem for $k=1$ and for simple families and simple transversals
was found in \cite{tour}, Theorem 6.2 but failed to mention (iii).
Note also that (iii) turns trivial for $M=k=1$, as the matrix in question has
a single column.
Conclusion (i) for the special
case  $L=k=M=1$ restricted to simple families is known as the strict Sperner theorem, already known
to Sperner \cite{sperner}; for  $M=1$, $L>1$, it was discovered by Paul Erd\H os
 \cite{erdos}.
However, Theorem~\ref{ujBLYM=} does not hold for $k=M\geq 2$, as the following example shows.
\begin{example}\label{example1} {\rm Let $k=M\ge 2$ and $L_{[M]}=1$ with
$|X_i|=m_i$ for $i\in[M]$, and assume $m_M\ge 2$.
For  integers $r,s$ with $1\leq r\leq m_M-1$ and $2\leq s\leq \min\left(n_1,\ldots,n_{M-1},\binom{m_M}{r}\right)$, 
consider a partition
 $\binom{X_M}{r}=\BB_1\uplus\ldots\uplus \BB_s$; and for each $j\in[M-1]$, fix an $s$-element set
$\{i_1^{(j)},\ldots,i_s^{(j)}\}\subseteq [n_j]^{\star}$. Define
a  $k$-dimensional $k$-part Sperner family $\FF$ as follows:
$$\FF= \biguplus_{\ell=1}^{s} \left(\left(\prod_{j=1}^{M-1}\binom{X_j}{i_{\ell}^{(j)}}\right)\times\BB_{\ell}\right).$$
This $\FF$ is not homogeneous, but by
\begin{eqnarray*}
\sum_{(i_1,\ldots,i_M)\in \pi_M}
\frac{p_{i_1,\ldots,i_M}}{\prod\limits_{j=1}^M \binom{m_j}{i_j}}
=\sum_{\ell=1}^s\frac{|\BB_{\ell}|\prod\limits_{j=1}^{M-1}\binom{m_j}{i_{\ell}^{(j)}}}{\binom{m_M}{r}\prod\limits_{j=1}^{M-1}\binom{m_j}{i_{\ell}^{(j)}}}=\frac{\sum\limits_{\ell=1}^s |\BB_{\ell}|}{\binom{m_M}{r}}=1=L_{[M]},
\end{eqnarray*}
$\FF$  still satisfies {\rm (\ref{allBLYM})} with (a single) equality. }
\end{example}
The above example can be easily extended to $L_{[M]}>1$. Although we did not characterize
cases of equality in (\ref{allBLYM}) for $k=M$,  in the case  $L_{[M]}=1$ 
we are able to give  a necessary condition for  an $M$--dimensional
 $M$-part Sperner family  to satisfy equality in (\ref{allBLYM}).
\begin{theorem}\label{lemma:k=M} 
Let $\FF^{\prime}$ be a $k$-dimensional $M$-part Sperner family with $k=M$ and  $L_{[M]}=1$, satisfying the equality
\begin{equation}\label{k=M:equality}
\sum_{E\in\FF^{\prime}}\frac{1}{\prod\limits_{i=1}^k \binom{m_i}{|E\cap X_i|}}=1.
\end{equation}
 Then for each $ i\in [M]$,
 the trace $\FF^{\prime}_{X_i}:=\{ F\cap X_i: F\in \FF^{\prime}\}$ of $\FF^{\prime}$ on $X_i$ is a union of full levels of $2^{X_i}$.
\end{theorem}
For the proof of Theorem~\ref{ujBLYM=} we need to prove a special case that is also
a straightforward generalization of the BLYM for $1$-part $L$-Sperner families, as stated below.
\begin{lemma}\label{multiset} Let $\FF$ be a multi-family of subsets of $[n]$ containing no multichain of length $L+1$. 
Then we have 
 \begin{equation}\label{eq:multi-BLYM} 
\sum_{F\in\FF}\frac{1}{\binom{n}{|F|}}\leq L,
\end{equation}
with equality if and only if $\FF$ is homogeneous.
\end{lemma}
\begin{proof}
The inequality part follows from Theorem~\ref{ujBLYM}, $k=M=1$. Suppose now we have equality in (\ref{eq:multi-BLYM}).   
We claim then that  $\FF$  can be partitioned into $L$ or less antichains. (In fact, this
is the multiset analogue of the well-known dual version of Dilworth's Theorem.) We now
mimick the inductive proof that works for simple families. 
For $L=1$, $\FF$ has to be a simple family and the claim is exactly the strict Sperner Theorem. 
Let $L>1$ and assume that the statement is true for all $1\le L^{\prime}<L$.
Consider  the set $\FF_1$ of maximal elements in $\FF$
(note that the multiplicity of each element in $\FF_1$ is one). Then   $\FF_2:=\FF\setminus\FF_1$ 
contains no multichain of length $L$. Thus we have
\begin{equation}
\sum_{F\in \FF_1}\frac{1}{\binom{n}{|F|}}\le 1\text{ and }\sum_{F\in\FF_2}\frac{1}{\binom{n}{|F|}}\le L-1.
\label{eq:ind} 
\end{equation}
But we also have
$$
L=\sum_{F\in \FF}\frac{1}{\binom{n}{|F|}}=\sum_{F\in \FF_1}\frac{1}{\binom{n}{|F|}}+\sum_{F\in\FF_2}\frac{1}{\binom{n}{|F|}},$$
therefore equality holds in both inequalities at (\ref{eq:ind}), and by the induction hypothesis both $\FF_1$ and $\FF_2$ are
homogeneous. The lemma follows.
 \end{proof}

\section{Convex hull of profile matrices of $M$-part multi-families}
\label{hull} 
The vertices of the convex hull of profile matrices of different kind of families were described by
P. L. Erd\H os, Frankl, and Katona
  \cite{efk}, facilitating the optimization of linear functions of the entries of profile matrices of members
  of the family in question.  P. L. Erd\H os and Katona  \cite{more-part} adapted the method 
  for $M$-part   Sperner families, and recently Aydinian,  Czabarka, P. L. Erd\H os, and  Sz\'ekely
  adapted it 
for 1-dimensional $M$-part $(m_1,\ldots,m_M; L_1,\ldots,L_M)$  Sperner families.  The purpose 
of this section to generalize these results for $k$-dimensional $M$-part  Sperner multi-families, and even further.

  Let $X=X_1 \uplus X_2 \uplus \cdots \uplus X_M$ be a partition of
the $n$-element underlying set $X,$ where $|X_i|=m_i\geq 1$ and
$m_1+ \ldots+ m_M=n$. Let $\FF $ be a multi-family of subsets of $X$. The
profile matrix $\Pp(\FF):=(p_{i_1,\ldots,i_M})_{(i_1,\ldots,i_M)\in  \pi_M  }$
can be   identified with a point or its  location vector in  the Euclidean space 
 $\mathbb{R}^N$, where $N=\prod_{j=1}^M n_i$.

Let $\alpha \subseteq \mathbb{R}^N$ be a finite point set. Let
$\langle \alpha \rangle$ denote  the {\sl convex hull} of the point
set, and  $\varepsilon(\alpha)
=\varepsilon( \langle \alpha\rangle)$  its {\sl extreme points}. It is
well-known that $\langle \alpha \rangle$  is equal to the set of all
convex linear combinations of its extreme points.

Let $\AAA $ be a family of multi-families of subsets of $X$. Let  $\mu
(\AAA)$  denote the set of all profile-matrices of the multi-families in $\AAA,$
i.e. $$\mu(\AAA)=\{\Pp(\FF): \FF\in \AAA\}.$$ Then the extreme points
$\varepsilon(\mu (\AAA))$ are integer vectors and they are profile
matrices  of   multi-families from $\AAA$.

In \cite{more-part}, P. L. Erd\H os and G.O.H. Katona developed 
a general method to determine the extreme points  $\varepsilon(\mu
(\AAA))$ for families of simple families. We adapt their results to a more general
setting. Let $I\Subset \pi_M$.
Let  $T(I)$
denote  the $M$-dimensional matrix, in which the entry
$t_{i_1,\ldots,i_M}(I)=\#[(i_1,\ldots,i_M),I]$ .
Furthermore, let $S(I)$ be the $M$-dimensional matrix, in which
$S_{i_1,\ldots,i_M}(I)=t_{i_1,\ldots,i_M}(I) { \binom{m_1}{ i_1}}\cdots {\binom{m_M}{ i_M }}$.
\old{
\begin{equation}\label{weight}
  S_{i_1,\ldots,i_M}(I)=\left\{%
\case { \binom{m_1}{ i_1}}\cdots \binom{m_M}{i_M }
      &  (i_1,\ldots,i_M)\in I \\
     0 &  (i_1,\ldots,i_M)\not\in I, \\
\end{array}\right.
\end{equation}
and let
\begin{displaymath}
  I(\FF):=\Big \{ (i_1,\ldots,i_M) \in \Pi
\  :\  P_{i_1,\ldots,i_M}(\FF) \ne 0 \Big \}.
\end{displaymath}
}
Recall that  a
multi-family of subsets of an $M$-partitioned underlying set  is
called {\sl  homogeneous}, if for any set, the sizes of its
intersections with the partition classes  already determine the (possibly $0$) multiplicity
with which
the set belongs to the multi-family.  It is easy to see that
 a homogeneous multi-family $\FF$ on 
$X$ has $\Pp(\FF) = S(I)$ for a certain multiset $I\Subset\pi_M$.

  We say that $\mathfrak{L}$  is a {\sl product-permutation} of $X$, if
the ordered $n$-tuple $\mathfrak{L}=(x_1,\ldots,x_n)$ is a permutation of
 $X=X_1\uplus X_2\uplus \cdots \uplus X_M$ such that
$X_j=\left \{x_i\ : \ i=m_1+\cdots+m_{j-1}+1, \ldots, m_1+\cdots+m_j
\right \}$ i.e. is $\mathfrak{L}$  is a juxtaposition of
permutations of $X_1$, $X_2$,\ldots,$X_M$, in this order. Furthermore,
we say that a subset $H \subseteq X$ is {\sl initial with respect}
to $\LL$, if for all $j=1,2,\ldots,M$ we have
\begin{displaymath}
H \cap X_j = \left\{x_{m_1+\cdots+m_{j-1}+1},\ldots,x_{m_1+ \cdots
+m_{j-1}+ |H \cap X_j|}\right \},
\end{displaymath}
i.e. $H\cap X_j$ is an initial segment in the permutation of $X_j.$
For an  $\HH$ multi-family on $X$, define  $\HH(\LL)=\multiset{ H^{\#[H,\HH]} : H \hbox{ is initial with respect to  } \LL}$.
Similarly, for  an $\AAA$ family of multi-families on $X$, 
let $\AAA(\LL):=\{\HH(\LL) : \HH\in \AAA\}.$ 

\begin{lemma}[cf. \cite{more-part} Lemma 3.1]
\label{th-blow}
Suppose that  for  a finite family $\AAA $ of $M$-part multi-families the set $\mu(\AAA (\LL))$
does not depend on the choice of $\LL.$ Then
\begin{equation}\label{eq-blow}
\mu (\AAA) \subseteq \left \langle \Big \{ S(I) :  I\Subset\pi_M \hbox{  with }     T(I)\in 
\mu(\AAA(\LL) )    \Big \}\right\rangle
\end{equation}
holds.
\end{lemma}
The next theorem follows easily from this lemma:
\begin{theorem}[cf. \cite{more-part} Theorem 3.2]\label{th-blow2}
Suppose that  a finite family $\AAA $ of $M$-part multi-families satisfies the following two
conditions:
\begin{equation}\label{eq-blow2}
\hbox{the\ set\ }
\mu(\AAA (\LL)) \hbox{
does not depend on } \LL, \hbox{ and }
\end{equation}
\begin{equation}\label{eq-blow3}
\text{\ for all\ }I\Subset\pi_M, \  T(I) \in \mu(\AAA(\LL)) \mbox{
implies } S(I) \in
  \mu(\AAA).
\end{equation}
  Then
\begin{equation}\label{eq-blow4}
  \varepsilon(\mu(\AAA)) =\varepsilon\biggl(\Bigl\{ S(I) : 
    I\Subset\pi_M , T(I)\in 
\mu(\AAA(\LL)) 
  \Bigl\} \biggl).
  \end{equation}
Consequently, among the maximum size elements of $\AAA$, there are homogeneous ones, and the profile matrices of maximum size elements of $\AAA$ are convex linear combinations of the profile matrices of homogeneous maximum size elements.
\end{theorem}
\begin{proof}
The identity 
$$\langle \mu (\AAA)\rangle  = \left \langle \Big \{ S(I) :  I\Subset\pi_M \hbox{  with }     T(I)\in
\mu(\AAA(\LL) )    \Big \}\right\rangle$$
follows from (\ref{eq-blow}) and  (\ref{eq-blow3}). If two convex sets are equal, then so are their extreme points.
\end{proof}

For any finite set $\Gamma$,  a {\sl $\Gamma$-multiplicity constraint} $\MMM_{\Gamma}$ is
$$\MMM_{\Gamma}=\{(A_{\gamma}\ge 0,\{\alpha^{\gamma}_{(i_1,\ldots,i_M)}\ge 0:(i_1,\ldots,i_M)\in\pi_M\}):\gamma\in\Gamma\}.
$$
We say that a multiset $\FF\Subset X$ satisfies the $\Gamma$-multiplicity constraint $\MMM_{\Gamma}$, if
\begin{equation*} 
\forall \gamma\in \Gamma \hbox{\ \ } \sum_{ (i_1,\ldots,i_M)\in \pi_M }  \alpha^\gamma_{i_1,\ldots,i_M}\cdot \max\{\#[F,\FF]: \forall j\in[M]\,\, |F\cap X_{j}|=i_j\} \leq A_\gamma.
\end{equation*} 
Analogously, a multiset $I\Subset\pi_M$ satisfies the $\Gamma$-multiplicity constraint $\MMM_{\Gamma}$, if
\begin{equation*}
\forall \gamma\in \Gamma \hbox{\ \ } \sum_{ (i_1,\ldots,i_M)\in \pi_M }  \alpha^\gamma_{i_1,\ldots,i_M}\cdot\#[(i_1,\ldots,i_M),I]\leq A_\gamma.
\end{equation*} 
It is easy to see that simple  families can be characterized  by the following condition: For all $(i_1,\ldots,i_M)\in\pi_M$, $\max\{\#[F,\FF]:\forall j |F\cap X_j|=i_j\}\leq 1$. This in turn can be written in the form of  a $\Gamma$-multiplicity constraint
by $\Gamma=\pi_M, A_{\gamma}=1, \alpha^\gamma_\lambda=\delta_{\gamma, \lambda}$ using
the Kronecker $\delta$ notation.
\begin{theorem} \label{hullappl}
To the family $\AAA$ of $k$-dimensional $M$-part Sperner multi-families with parameters $L_P$  for $P\in \binom{[M]}{k}$ satisfying 
a  fixed $\Gamma$-multiplicity constraint $\MMM_{\Gamma}$, Theorem~\ref{th-blow2} applies. In other words, all extreme points of $\mu(\AAA)$ come
from homogeneous multi-families.
\end{theorem}
This theorem implies the results of \cite{more-part} and \cite{tour} on the convex hull with one exception: there not just all extreme points came from homogeneous families, but all homogeneous families provided
extreme points. This is not the case, however, for multi-families, but characterizing which homogeneous families are extreme is hopeless. For {\sl simple families}, however, we can characterize these extreme points.  

We say that an $I$ 
$k$-dimensional $M$-part multi-transversal with $\Gamma$-multiplicity constraint $\MMM_{\Gamma}$ is 
{\sl lexicographically maximal (LEM)}, 
 if the
support set  $\supp(I)\subseteq\pi_M$ of the multiset $I$ has an ordering $\vec{j}_1,\vec{j}_2,\ldots,\vec{j}_s$, such that  for every  $I^{\star}$ 
$k$-dimensional $M$-part multi-transversal with $\Gamma$-multiplicity constraint $\MMM_{\Gamma}$, the following holds:\\
(i) If $\vec{j}_1\in \supp(I^{\star})$, then $\#[\vec{j}_1,I]\geq \#[\vec{j}_1,I^{\star}]$, and \\
(ii) for every $1\le \ell\le s-1$, if $\{\vec{j}_1,\vec{j}_2,\ldots,\vec{j}_\ell\}\subseteq
\supp(I^{\star})$ and $\#[\vec{j}_h,I]= \#[\vec{j}_h,I^{\star}]$
for $h=1,2,\ldots,\ell$, then $\#[\vec{j}_{\ell+1},I]\geq \#[\vec{j}_{\ell+1},I^{\star}]$.
\begin{lemma}\label{lemma:LEM} 
For a family $\AAA$ of $k$-dimensional $M$-part Sperner multi-families with parameters $L_P$  for $P\in \binom{[M]}{k}$ satisfying 
a  $\Gamma$-multiplicity constraint $\MMM_{\Gamma}$, all the profile matrices $S(I)$, where  $I\Subset\pi_M$ is a LEM  $k$-dimensional multi-transversal with the same $\Gamma$-multiplicity constraint, are extreme points of $\mu(\AAA)$.
\end{lemma}
For simple $k$-dimensional $M$-part Sperner families $\FF$, i.e. when the $\Gamma$-multiplicity constraint $\MMM_{\Gamma}$ {\sl includes} the conditions
$\max\{\#[F,\FF]:\forall j\,\,|F\cap X_j|=i_j\}\leq 1$ for all $(i_1,\ldots,i_M)\in \pi_M$, every $I$ $k$-dimensional $M$-part Sperner multi-transversal with parameters $L_P$  for $P\in \binom{[M]}{k}$ satisfying 
a  $\Gamma$-multiplicity constraint $\MMM_{\Gamma}$ has the LEM property. This finally derives the convex hull results of  
\cite{tour} and \cite{more-part} from our results. Note, however,
that the $\Gamma$-multiplicity constraint provides new results even for the classical $M=1$ case. For completeness, we state explicitly our result for simple families.
\begin{theorem} \label{convex}
The extreme points of the convex hull of  profile matrices of all $k$-dimensional
$M$-part  simple Sperner families with a $\Gamma$-multiplicity constraint  $\MMM_{\Gamma}$
are exactly
the profile matrices of the homogeneous families corresponding to $k$-dimensional
$M$-part  simple transversals with the same $\Gamma$-multiplicity constraint $\MMM_{\Gamma}$. 
Therefore,
among the maximum size $k$-dimensional
$M$-part Sperner families   with a $\Gamma$-multiplicity constraint, there are homogeneous ones.
\end{theorem}

\section{Applications of the convex hull method}

Although the previous section reduces the
problem of finding the maximum size of such families to a "number" problem from
a "set" problem, however, we assert that the problem is still "combinatorial" due to the
complexity of transversals:
\begin{problem}\label{problem1} For a $(t_1,\ldots,t_M)\in \pi_M$, set the {\sl weight} $W(t_1,\ldots,t_M)=
\prod_{i=1}^M\binom{m_i}{t_i}.$ Find a set of codewords $C\subseteq\pi_M$ with the largest possible
sum of weights,
 such that for
every  $P\in\binom{[M]}{k}$ set of character positions, if the characters are prescribed
in any way for the $i\notin P$ character positions, at most $L_P$ from $C$
show all the prescribed values.
\end{problem}

In view of Theorem \ref{convex}, Problem~\ref{problem1} is equivalent to finding maximum size $k$-dimensional
$M$-part  simple Sperner families.
Recall  that this problem is not solved even for the case $L=1, k=1, M\geq 3$
(see \cite{tour} for a survey of results).
Note also that there are examples in \cite{tour} without a
full $1$-dimensional transversal defining  a maximum size homogeneous family,
unlike in the case $M=2,L=1$.

Our results allow us to prove that certain maximum size families must always be homogeneous. 
\begin{theorem}\label{genhom} Let $1\le k<M$ or $k=M=1$.
If every maximum size homogeneous $k$-dimensional $M$-part Sperner family (alternatively: Sperner multi-family)
satisfies
{\rm (\ref{allBLYM})} with equality, then every maximum size $k$-dimensional $M$-part Sperner family 
(Sperner multi-family) is
homogeneous.
\end{theorem}
\begin{proof}
Fix a $P\in\binom{[M]}{k}$ and let $C=L_P\prod_{j\notin P}n_j$. By the assumptions, the value of $C$ is independent of $P$. 
Let $\FF$ be a maximum size family/multi-family with profile matrix
$\Pp(\FF)=(p_{(i_1,\ldots,i_M)})$. Let $\GG_1,\ldots,\GG_s$ be en enumeration of all maximum size homogeneous families/multi-families, and
let $I_1,\ldots,I_s\subseteq\pi_M$ be the $(M-1)$-dimensional transversals/multi-transversals for which 
$\Pp(\GG_j):=(p_{(i_1,\ldots,i_M)}^{(j)})=S(I_j)$. By the assumptions for each $j\in[s]$ we have
$$\sum_{(i_1,\ldots,i_M)\in\pi_M}\frac{p_{(i_1\ldots,i_M)}^{(j)}}{\prod_{\ell=1}^M\binom{m_{\ell}}{i_{\ell}}}=C.
$$
By Theorems~\ref{hullappl} and \ref{convex} we have $\lambda_j\ge 0$ such that
$\sum_j\lambda_j=1$ and $\Pp(\FF)=\sum_{j=1}^s\lambda_j\Pp(\GG_j)$. Therefore
\begin{eqnarray*}
\sum_{(i_1,\ldots,i_M)\in\pi_M}\frac{p_{(i_1\ldots,i_M)}}{\prod_{\ell=1}^M\binom{m_{\ell}}{i_{\ell}}}
&=&\sum_{(i_1,\ldots,i_M)\in\pi_M}\frac{\sum\limits_{j=1}^{s}\lambda_jp_{(i_1\ldots,i_M)}^{(j)}}{\prod_{\ell=1}^M\binom{m_{\ell}}{i_{\ell}}}\\
&=&\sum_{j=1}^{s}
\left(\lambda_j\sum_{(i_1,\ldots,i_M)\in\pi_M}\frac{p_{(i_1\ldots,i_M)}^{(j)}}{\prod_{\ell=1}^M\binom{m_{\ell}}{i_{\ell}}}\right)\\
&=&C\sum_{j=1}^s\lambda_j=C,
\end{eqnarray*}
and $\FF$ is homogeneous by Theorem~\ref{ujBLYM=}.
\end{proof}

We state some simple results for the case when all parameters $L_P=1$.
\begin{theorem} Consider the (simple) $M$-part families such that
for all $E,F\in\FF$, if $E\ne F$ then there is a $j\in[M]$ such that $E\cap X_j\ne F\cap X_j$. If $\FF$ is maximum size among these families, then
\begin{equation}\label{maxk=M}
|\FF|=\prod_{i=1}^M\binom{m_i}{\lfloor m_i/2\rfloor}.
\end{equation}
$\FF$ is a maximum size homogeneous family precisely when 
$\Pp(\FF)=S((\ell_1,\ldots,\ell_M))$ where for each $i\in[M]$, $\ell_i\in\{\lfloor m_i/2\rfloor,\lceil m_i/2\rceil\}$. In particular, when all $m_i$ are even, the maximum size family is unique and homogeneous.
\end{theorem}
\begin{proof} A family $\GG$ satisfies the conditions precisely when it is a $M$-dimensional $M$-part Sperner family with $L=1$.
By Theorem~\ref{convex}, there are homogeneous families among such maximum size families. So let $\FF$ be a homogeneous
maximum size family. Then $\Pp(\FF)=S(I)$ for some $I\subseteq\pi_M$. It follows from the conditions that
$|I|=1$, so $I=\{(i_1,\ldots,i_M)\}$ and $|\FF|=\prod_{j=1}^{M}\binom{m_j}{i_j}$. (\ref{maxk=M}) follows, moreover the homogeneous
maximum size families are precisely the ones listed in the theorem.

Since by Theorem~\ref{convex} the profile matrix
$\Pp(\FF)$ is the convex combination of the profile matrices
of maximum size families, it follows that for $m_i$ even the maximum size family is unique.
\end{proof}
Note that it is easy to create a nonhomogeneous maximum size family when at least one of the $m_i$ is odd along the lines of Example~\ref{example1}.

For the next result we will use the following, which follows easily by induction on $K$.
\begin{lemma}\label{nagyot} 
Let $K,M$ be positive integers and for each $i\in[K]$ and $j\in[M]$ let $a_{ij}$ be nonnegative reals such that
$a_{1j}\ge a_{2j}\ge\cdots\ge a_{K,j}$ and $S_K$ denotes the set of permutations on $[K]$. Then
$$
\max\{\sum_{\ell=1}^{K}\prod_{j=1}^{M} a_{\pi_j(\ell),j}:\forall j\in[M]\,\, \pi_j\in S_K\}=\sum_{\ell=1}^{K}\prod_{j=1}^{M}a_{\ell j}.
$$
\end{lemma}\hfill\qed

\begin{theorem} Assume that $m_M=\min m_i$ and consider the $(M-1)$-dimensional $M$-part
Sperner families with parameters $L_{[M]\setminus\{i\}}=1:i\in[M]$. If $\FF$ is of maximum size amongst these families, then
$$
|\FF|=\sum_{i=0}^{m_M}\prod_{j=1}^M\binom{m_j}{\lceil\frac{m_j}{2}\rceil+(-1)^i\lceil\frac{i}{2}\rceil}.
$$
Moreover, if $\FF$ is a maximum size homogeneous family, then $\Pp(\FF)=S(I)$ for some
$
I=\{(b_{i1},\ldots,b_{iM}):i\in[n_M]^{\star}\}$ where for each fixed $j\in[M]$ the $b_{ij}$ are $n_M$ different integers
from $[n_j]^{\star}$ such that $\binom{m_j}{b_{ij}}=\binom{m_j}{\lceil\frac{m_j}{2}\rceil+(-1)^i\lceil\frac{i}{2}\rceil}$.

If in addition $m_1=\ldots=m_M$, then all maximum size families are homogeneous.
\end{theorem}

\begin{proof} Theorem~\ref{convex} implies that amongst the maximum size families there are homogeneous ones. Let
$\FF$ be a (not necessarily maximum size)
homogeneous $(M-1)$-dimensional $M$-part Sperner family with all parameters $1$, and let $I$
be the transversal for which $\Pp(\FF)=S(I)$. Then if $\vec{i}=(i_1,\ldots,i_M)$ and 
$\vec{i}^{\prime}=(i^{\prime}_1,\ldots,i^{\prime}_M)$ are
elements of $I$ such that for some $\ell\in[M]$ $i_{\ell}=i^{\prime}_{\ell}$, we must have that $\vec{i}=\vec{i}^{\prime}$.
Therefore there is a $K\le n_M$ such that
$I=\{(b_{i1},\ldots,b_{iM}):i\in[K]^{\star}\}$ where for each fixed $j\in[M]$ the $b_{ij}$ are $n_M$ different integers
from $[n_j]^{\star}$ and $|\FF|=\sum_{\ell=1}^K\prod_{j=1}^M\binom{m_j}{b_{\ell j}}$. The statement about maximum size homogeneous families follows from 
Lemma~\ref{nagyot} and the fact that
$$\binom{m_j}{\lceil\frac{m_j}{2}\rceil+(-1)^0\lceil\frac{0}{2}\rceil}\ge \binom{m_j}{\lceil\frac{m_j}{2}\rceil+(-1)^1\lceil\frac{1}{2}\rceil}
\ge\cdots\ge \binom{m_j}{\lceil\frac{m_j}{2}\rceil+(-1)^{m_j}\lceil\frac{m_j}{2}\rceil}.$$
The rest follows from Theorem~\ref{genhom}.
\end{proof}

\section{New $k$-dimensional transversals and mixed orthogonal arrays}
\label{newMOA}
 Aydinian,  Czabarka,   Engel, P. L. Erd\H{o}s, and  Sz\'ekely \cite{acees} ran into
 MOAs  as they faced the problem of  constructing 1-dimensional full transversals
 for $M>2$. Using the indicator function of the $k$-dimensional  transversal in
 (\ref{korlat}) instead of the transversal itself, it is easy to see that the existence of "fractional full $k$-dimensional  transversal" is trivial. Therefore the construction
 problem of full $k$-dimensional transversals
 is a problem of integer programming.
 For $M=2$, such construction was found \cite{fgos}
 using matching theory, which
 does not apply for $M>2$. \cite{acees} observed Proposition \ref{elemi} for $k=1$ (the property "simple" was assumed tacitly) and constructed
 1-dimensional full transversals for any parameter set,
 and infinitely many MOAs with constraint $M$ and strength $M-1$.
The key element of the construction was the  elementary
 Lemma~\ref{1dim}, which only uses properties of 
the fractional part
$\fracpart{x}=x-\lfloor x\rfloor$ function
 of a real number $x$.   This lemma will
be heavily used again in this paper.

\begin{lemma}\label{1dim} {\rm [Engel's Lemma.]} Let $n$ be a positive integer, $\mu,\alpha,\beta$ be real numbers
such that $0<\mu$ and $0\le\beta\le 1-\mu$. Then
\begin{equation*}
\Bigg\vert\left\{i\in[n]^{\star}: \fracpart{\alpha+\frac{i}{n}}\in[\beta,\beta+\mu)\}
\right\}\Bigg\vert
\in\{\lfloor \mu n\rfloor, \lceil\mu n\rceil\}.
\end{equation*}
\end{lemma}

All our constructions for full $k$-dimensional transversals and simple MOAs are based on the following construction.
\begin{construction} For $n_1,\ldots,n_M$ positive integers,  $0<\mu\le 1$  real,  and $0\leq \beta\le 1-\mu$, define
\begin{equation}
\mathbb{C}(n_1,\ldots,n_M;\beta,\mu):=
\left\{
(i_1,\ldots,i_M)\in\pi_M:\fracpart{\sum_{j=1}^M\frac{i_j}{n_j}}\in[\beta,\beta+\mu)
\right\}.\label{construction}
\end{equation}
\end{construction}
For the case $k=1$, \cite{acees} showed that for any
 $i\in[M]$, any $L_i\in[n_i]$, 
any $0<\mu\le\min\{\frac{L_i}{n_i}: i\in[M]\}$, any  $0\leq \beta\leq 1-\mu$, the
construction in (\ref{construction}) is a $1$-dimensional transversal for the given
parameters, moreover, if $\mu=\min\{\frac{L_i}{n_i}: i\in[M]\}$, then this $1$-dimensional
 transversal is full.

The following facts are almost immediate from the construction:
\begin{proposition}\label{egyszeru} Let $n_1,\ldots,n_M$ be positive integers, $k\in[M]$, and $\{L_P:P\in\binom{[M]}{k}\}$
be given such that $1\le L_P\le \prod_{i\in P} n_i$ are integers. If there is a $0<\mu_0\le 1$
such that for each $0\le \beta\le 1-\mu_0$ the construction
$\mathbb{C}(n_1,\ldots,n_M;\beta,\mu_0)$ is a
full $k$-dimensional  transversal with these $L_P$ parameters, then
\begin{enumerate}[{\rm (i)}]
\item $\mathbb{C}(n_1,\ldots,n_M;\beta,\mu)$ is a $k$-dimensional transversal with these
parameters for every $0<\mu<\mu_0$ and $0\le\beta\le 1-\mu$.
\item $\pi_M$ can be partitioned into $\lceil\frac{1}{\mu_0}\rceil$ $k$-dimensional
transversals with these parameters, and $\lfloor\frac{1}{\mu_0}\rfloor$ of these are full.
\item With $\alpha=\min\left\{\frac{L_P}{\prod_{i\in P} n_i}: P\in\binom{[M]}{k}\right\}$, we have $\lfloor\frac{1}{\alpha}\rfloor\le \frac{1}{\mu_0}\le\lceil\frac{1}{\alpha}\rceil$. In particular, if $\frac{1}{\alpha}$ is an integer, all $k$-dimensional
transversals in the partition in (ii) are full.
\end{enumerate}
\end{proposition}
\begin{proof} (i) follows from the fact that for every $\beta,\mu$ in (i), exists a
$0\leq \beta^{\prime}\leq 1-\mu_0$, such that $[\beta, \beta+\mu)\subseteq [\beta^{\prime}, \beta^{\prime}+\mu_0)$.
(Here we did not use  the fullness in the hypothesis.) For (ii), we use the fact that
$[0,1)$ can be partitioned into $\lceil\frac{1}{\mu_0}\rceil$ half-open intervals,
$\lfloor\frac{1}{\mu_0}\rfloor$ of which has length $\mu_0$. Finally, (iii) follows from
(\ref{korlat2}) and (ii).
\end{proof}

We arrived at the following generalization of Engel's lemma (Lemma~\ref{1dim}):
\begin{lemma}\label{Engelprime} Let $n_1,\ldots,n_k$ be positive integers,
$N=\lcm(n_1,\ldots,n_k)$, $K=\prod_{i=1}^{k}n_k$ and $\ell=\frac{K}{N}$. If
$\alpha,\beta,\mu$ are real numbers
with $0<\mu<1$ and $0\le \beta\le 1-\mu$, then
\begin{equation*}
\Bigg\vert\left\{(i_1,\ldots,i_k)\in\pi_M: \fracpart{\alpha+\sum_{j=1}^{k}\frac{i_j}{n_j}}\in[\beta,\beta+\mu)
\right\}\Bigg\vert
\in\left\{\ell\lfloor \mu N\rfloor,\ell\lceil\mu N\rceil\right\}.
\end{equation*}
\end{lemma}
The proof of Lemma~\ref{Engelprime} is postponed to Section~\ref{MOAproofs}.
Based on Lemma~\ref{Engelprime}, the following theorem gives a sufficient criterion
to use (\ref{construction}) to construct full $k$-dimensional
transversals. For $k=1$ it gives back the construction in \cite{acees}.

We set a generic
notation here for the rest of this section and Section~\ref{MOAproofs}.
Let us be given $n_1,\ldots,n_M\geq 1$ integers, a $k\in[M]$, and
for every
 $P\in\binom{[M]}{k}$ let the integer $L_{P}$ be  given such that
 $1\le L_P\le \prod_{i\in P} n_i.$ For every
 $P\in\binom{[M]}{k}$, set $K_P=\prod_{i\in P}n_i$, $N_P=\lcm\{n_i:i\in P\}$, and
 $\ell_P=\frac{K_P}{N_P}$.
 \begin{theorem}\label{fotetel}
 Assume that a $\mu>0$ is given such that
 \begin{equation} \label{gencond}
 \forall \hbox{\  } P\in\binom{[M]}{k} \hbox{\  \  \  } \ell_P\lceil \mu N_P\rceil \leq L_P.
 \end{equation}
Then for any
$0\le\beta\le 1-\mu$, $\mathbb{C}(n_1,\ldots,n_M;\beta,\mu)$ is a $k$-dimensional transversal with the
given parameters $L_P$. Moreover,
if $\mu=\min_{P\in\binom{[M]}{k}}\frac{L_{P}}{K_P}$, then it is a full transversal.
\end{theorem}
Note that condition (\ref{gencond}) easily implies that $\mu\le\min_{P\in\binom{[M]}{k}}\frac{L_{P}}{K_P}$. 
The proof of Theorem~\ref{fotetel} is also postponed to Section~\ref{MOAproofs}.
\begin{corollary}\label{firsttry}
If the $n_i$ numbers are pairwise relatively prime,
then for $\mu=\min_{P\in\binom{[M]}{k}}\frac{L_{P}}{K_P}$
and for any $0\le\beta\le 1-\mu$,  $\mathbb{C}(n_1,\ldots,n_M;\beta,\mu)$ is a full
$k$-dimensional transversal with these parameters.
\end{corollary}
\begin{proof} It is enough to check that (\ref{gencond}) holds in Theorem~\ref{fotetel}
for this $\mu$.
Let $P\in\binom{[M]}{k}$. From the fact that the $n_i$ numbers are relatively prime, it follows
that $K_P=N_P$ and $\ell_P=1$.
Therefore
$\ell_P\lceil\mu N_P\rceil\le\lceil\frac{L_P}{N_P}\cdot N_P\rceil=L_P$.
\end{proof}

Corollary~\ref{firsttry} ensures that we have a full $k$-dimensional transversal
for all $k\in[M]$ and all allowed settings of $\{L_P: P\in\binom{[M]}{k}\}$
whenever $n_1,\ldots,n_M$ are relatively prime. Unfortunately, this does not allow us to
chose parameters that give MOAs, i.e. for values of $L_P$ such that
$\frac{K_P}{L_P}$ is constant.
We can still use the construction in (\ref{construction}) to find such transversal, but we
need to put more restrictions on the possible values of the $L_P$.

\begin{corollary}\label{oarray}
Assume
that there is a constant $0<\mu\le 1$
such that
for each $P\in\binom{[M]}{k}$,
$\mu N_P$ is an integer and $L_P=\mu K_P$.
Then, for every $0\le \beta\le 1-\mu$,
 $\mathbb{C}(n_1,\ldots,n_M;\beta,\mu)$ is a full $k$-dimensional transversal, and
provides a simple  MOA of strength $M-k$.
\end{corollary}
\begin{proof}
The condition on $\mu$ gives
$\ell_P\lceil\mu N_P\rceil=\ell_P\mu N_P=\mu K_P=L_P$ and
$\mu=\min_P\frac{L_P}{K_P}$; the statement follows from Theorem~\ref{fotetel}.
\end{proof}

While the conditions of the theorem may at first glance
seem restrictive, we can easily satisfy them. For
a given $k\in[M]$ we  chose a sequence of
integers $ j_1, j_2, \cdots, j_M$, and set $n_i=\prod_{q=1}^i j_{q}$. Set
$q$ as one of the divisors of $n_k$ and $\mu=\frac{1}{q}$. It is clear that
this choice of $\mu$ satisfies the conditions of Theorem~\ref{oarray}, since
for each $P\in\binom{[M]}{k}$ we have that $N_P=\lcm\{n_i:i\in P\}=n_{\max P}$. By the the choice of the $n_i$'s and the fact that
$k\le\max P$, 
$n_k$ divides $N_P$. Since
$\mu n_k$ is an integer, so is $\mu N_P$.
Thus, for each $P\in\binom{[M]}{k}$ if we chose
$L_{P}=\mu K_P$, then the construction gives a simple MOA with the given parameters.

\bigskip

We also
provide two "generic" constructions to create new full $k$-dimensional multi-transversals
and MOAs from already known
 ones, under some numerical conditions: "linear combination", 
and "tensor product". The correctness of these constructions is straightforward from the definitions.

 \begin{proposition}\label{prop:lincomb} {\rm [Linear Combination for Transversals.]}
 \begin{enumerate}[{\rm (i)}]
 \item
Let $j\in\mathbb{Z}^+$ and for each $\ell\in[j]$ let $\TT_{\ell}$ be a $k$-dimensional multi-transversal on 
$\pi_M$ with parameters
$L^{(\ell)}_P:P\in\binom{[M]}{k}$. Assume that for all $\ell\in[j]$ positive reals $\alpha_{\ell}$ are given such that for all
$(i_1,\ldots,i_M)\in\pi_M$ the quantity $\sum_{\ell=1}^{j}\alpha_{\ell}\cdot\#[(i_1,\ldots,i_M),\TT_{\ell}]$ is an integer, and
let 
$$\TT^{\star}=\multiset{(i_1,\ldots,i_M)^{\sum_{\ell=1}^{j}\alpha_{\ell}\cdot\#[(i_1,\ldots,i_M),\TT_{\ell}]}:(i_1,\ldots,i_M)\in\pi_M}.$$
Then
$\TT^{\star}$ is a $k$-dimensional multi-transversal on $\pi_M$ with parameters 
$L^{\star}_P:=\lfloor\sum_{\ell=1}^j \alpha_{\ell} L^{(\ell)}_P\rfloor:~P\in\binom{[M]}{k}$.
 \item Assume further that each $\TT_{\ell}$ above is a full multi-transversal and there is a common 
 $A\in\binom{[M]}{k}$ on which all
 $\TT^{\ell}$ simultaneously meet the bound, i.e.
 $$\forall \ell\in[j]\,\,\,\,L_{A}^{\ell}\prod_{j\notin A} n_j=\min_{P\in\binom{[M]}{k}} \left(L_{P}^{(\ell)}\prod_{j\notin P} n_j\right).
 $$
 Then $\TT^{\star}$ is a full multi-transversal as well.\hfill\qed
 \end{enumerate}
\end{proposition}
Since the condition is true when the $\alpha_{\ell}$ are all integers, this means in particular that if $\TT_1$ and $\TT_2$ are both $k$-dimensional multi-transversals, then so is
$\TT_1\uplus\TT_2$.
 \begin{proposition} {\rm [Linear Combination for MOAs.]}
Let $j\in\mathbb{Z}^+$ and for each $\ell\in[j]$ let $\TT_{\ell}$ be a full $k$-dimensional multi-transversal on 
$\pi_M$ with parameters
$L^{(\ell)}_P:P\in\binom{[M]}{k}$ such that $L^{\ell}_P\cdot\prod_{j\notin P} n_j$ is independent of $P$ (i.e. $\TT_{\ell}$ is a MOA).
Let nonzero reals $\alpha_{\ell}$ be given for all $\ell\in[j]$ such that for all
$(i_1,\ldots,i_M)\in\pi_M$ the quantity $\sum_{\ell=1}^{j}\alpha_{\ell}\cdot\#[(i_1,\ldots,i_M),\TT_{\ell}]$ is a 
non-negative integer, and
let $\TT^{\star}=\sum_{\ell=1}^{j}\alpha_{\ell}\TT_{\ell}$ be defined as
$$\TT^{\star}=\multiset{(i_1,\ldots,i_M)^{\sum_{\ell=1}^{j}\alpha_{\ell}\cdot\#[(i_1,\ldots,i_M),\TT_{\ell}]}:(i_1,\ldots,i_M)\in\pi_M}.$$
Then
$\TT^{\star}$ is a full $k$-dimensional multi-transversal on $\pi_M$ with parameters 
$L^{\star}_P=\sum_{\ell=1}^j \alpha_{\ell} L^{(\ell)}_P:~P\in\binom{[M]}{k}$, moreover, $L^{\star}_P\prod_{j\notin P} n_j$ is independent of $P$ (with other words, $\TT^{\star}$ is a MOA).
\hfill\qed
\end{proposition}
In Proposition~\ref{prop:lincomb}, chose $j=2$, and MOAs $\TT_1$ and $\TT_2$ such that
$\#[\vec{i},\TT_2]\ge\#[\vec{i},\TT_2]$ for all $\vec{i}\in\pi_M$. Then setting $\alpha_{\ell}=(-1)^{\ell}$ for $\ell\in[2]$ satisfies
the conditions of Proposition~\ref{prop:lincomb} and
$\TT^{\star}=\sum_{\ell=1}^2\alpha_{\ell}\TT_{\ell}=\TT_2\setminus\TT_1$; this type of linear combination
is exactly the relative complementation on MOAs.
Accordingly, if a MOA contains another one with the same strength as a subarray, erasing the rows of the subarray results in a new MOA. 

Proposition~\ref{prop:lincomb}  allows us to use the construction in (\ref{construction}) to build 
simple MOAs different from the ones in (\ref{construction}).
\begin{corollary}
Let $n_1,\ldots,n_M$,  and
$0<\mu<1$ be given such that they satisfy the conditions of Corollary~\ref{oarray}.
For a fixed positive integer $Q$, and for each $i\in[2Q+1]$ let $\beta_i$ be given such that
$0\le \beta_1<\beta_2<\cdots<\beta_{2Q+1}<\beta_1+\mu\le 1$ and $\beta_{2Q+1}\le 1-\mu$. 
Define ${\cal I}\subseteq[0,1)$ by
$${\cal I}=\left(\bigcup_{\ell=1}^{Q}\Bigl[\beta_{2\ell-1},\beta_{2\ell}\Bigl)\right)\cup\Bigl[\beta_{2\ell+1},\beta_1+\mu\Bigl)
\cup
\left(\bigcup_{\ell=1}^{Q}\Bigl[\beta_{2\ell}+\mu,\beta_{2\ell+1}+\mu\Bigl)\right).$$
Then the following is a 
$k$-dimensional transversal on $\pi_M$ with parameters $L_P=\mu\prod_{q\in P} n_q$ and  provides a simple 
MOA of strength $M-k$:
$$
\TT=\left\{
(i_1,\ldots,i_M)\in\pi_M:\fracpart{\sum_{r=1}^M\frac{i_r}{m_r}}\in{\cal I}\right\}.
$$
\end{corollary}
\begin{proof} For $\ell\in[2Q+1]$ let
$\TT_{\ell}=\mathbb{C}(n_1,\ldots,n_M;\beta_{\ell},\mu)$. By Corollary~\ref{oarray}, each $\TT_{\ell}$ is a full $k$-dimensional transversal on $\pi_M$ with parameters $L_P=\mu\prod_{j\in P}n_j$ satisfying the conditions of Proposition~\ref{prop:lincomb}.
Also, using $\alpha_{\ell}=(-1)^{\ell+1}$ we obtain that $\TT=\sum_{\ell=1}^{2Q+1} \alpha_{\ell}\TT_{\ell}$. The statement
follows from Proposition~\ref{prop:lincomb} and the fact that $\sum_{\ell=1}^{2Q+1}\alpha_{\ell}=1$.
\end{proof}

\begin{proposition} {\rm [Tensor product.]}
 \begin{enumerate}[{\rm (i)}]
\item Assume that $\TT_1$  and $\TT_2$ are
$k$-dimensional multi-transversals on
$\prod_{j=1}^{M} [n_j^{(1)}]^{\star}$ and $\prod_{j=1}^{M} [n_j^{(2)}]^{\star}$ with parameters
 $L^{(1)}_P$ and $L^{(2)}_P$
($P\in \binom{[M]}{k}$), respectively. Then 
$$\TT=\Bmultiset{
(a_1n_1^{(2)}+b_1,\ldots,a_Mn_M^{(2)}+b_M)^{\#[(a_1,\ldots,a_M),\TT_1]\cdot\#[(b_1,\ldots,b_M),\TT_2]}}$$
is a $k$-dimensional multi-transversal on
$\prod_{j=1}^{M} [n_j^{(1)}n_j^{(2)}]^{\star}$ with parameters $L_P=L^{(1)}_PL^{(2)}_P$.
\item Assume that $\TT_1$  and $\TT_2$ above are full multi-transversals,
 and assume that there exists an $A\in\binom{[M]}{k}$, in which both meet the
 bound set by $A$, i.e. for $i\in\{1,2\}$ we have
$$
L_{A}^{(i)}\prod_{j\notin A} n_j^{(i)}=\min_{P\in\binom{[M]}{k}} \left(L_{P}^{(i)}\prod_{j\notin P} n_j^{(i)}\right).
$$
Then  $\TT$ is a  full multi-transversal as well.\hfill\qed
\end{enumerate}
\end{proposition}
Condition (ii) holds, in particular, if (\ref{konstant}) holds for both $\TT_1$
and $\TT_2$, therefore the tensor product of MOAs
of the same constraint and the same strength is a MOA of the same constraint and the same strength, 
using in the $i^{\text{th}}$ column of $\TT$ the
Cartesian product of the symbol sets of the $i$-th columns of $\TT_1$
and $\TT_2$ with appropriate multiplicities.

\section{Proofs to the Sperner type results}
\label{Sproofs}
In the proofs of this section we will frequently make use of the following structure. Let
$\FF$ be a multi-family on $\biguplus_{i\in[M]} X_i$. Fix a $D\subseteq [M]$ and let
$F\subseteq X\setminus \biguplus_{i\in D} X_i$. We define
\begin{equation*}
\FF(F;D)=\Bmultiset{(E\setminus F)^{\#[E,\FF]}: E\cap\biguplus_{i\in[M]\setminus D} X_i=F }.
\end{equation*}
The following are clear from the definitions.
\begin{lemma}\label{lemma:FF(F)}
Let $\FF$ be a $k$-dimensional $M$-part  Sperner multi-family with parameters $L_P: P\in\binom{M}{k}$. Fix
$k\le N\le M$,  a $D\in\binom{[M]}{N}$ and let
$F\subseteq X\setminus \biguplus_{i\in D} X_i$. 
 \begin{enumerate}[{\rm (i)}]
 \item $\FF(F;D)$ is a $k$-dimensional $N$-part Sperner multi-family on $\biguplus_{i\in D} X_i$
 with parameters $L_P:P\in\binom{D}{k}$. 
 \item If $\FF$ is a simple family, so is $\FF(F;D)$.\hfill\qed
\end{enumerate}
\end{lemma}
\noindent{\sl Proof to Theorem~\ref{ujBLYM}:}
First assume $M=k$, and call our multi-family $\FF^{\prime}$
instead of  $\FF$. For each $i\in [M]$, there are $m_i!$ (simple) chains of maximum size (i.e. of length  $n_i$) in $X_i$.  We count the number of ordered $(k+1)$-tuples in the following multiset
in two ways:
\begin{equation*}
\Bmultiset{(E,\CC_1,\ldots,\CC_k)^{\#[E,\FF^{\prime}]}:\,\,E\in\prod_{i\in[M]}\CC_i, \hbox{\  where \ } \CC_i \hbox{\  is  chain of size $n_i$ in\ } X_i}.
\end{equation*}
Since each
chain product $\prod_{i=1}^M \CC_i$ contains at most $L_{[M]}$ sets from $\FF^{\prime}$ by definition,
the number of such $(k+1)$-tuples
is at most $L_{[M]}\prod_{i\in [M]} m_i!$. Since
each $E\in \FF^{\prime}$ can be extended to precisely $\prod_{i\in [M]} |E\cap X_i|!(m_i-|E\cap X_i|)!$  chain products with each chain being maximum size,
we have that
\begin{eqnarray*}
\sum_{E\in \FF^{\prime}} \prod_{i\in [M]} |E\cap X_i|!(m_i-|E\cap X_i|)!\le L_{[M]}\prod_{i\in [M]}m_i!
\end{eqnarray*}
from which the claimed inequality follows in the form
$$
\sum_{E\in \FF^{\prime}}\frac{1}{\prod\limits_{i\in [M]}\binom{m_i}{|E\cap X_i|}}\le L_{[M]}.
$$
Now assume $M>k$ and
take an arbitrary  $P\in\binom{[M]}{k}$ to prove the theorem for our multi-family $\FF$.
Take an  $F \subseteq X\setminus\bigcup_{i\in P}X_i$, and assume
$f_i=|F \cap X_i|$ for $i\notin P$. By Lemma~\ref{lemma:FF(F)}, $\FF(F;P)$ is a $k$-dimensional  $k$-part 
Sperner multi-family with parameter $L_P$, and therefore, using Theorem~\ref{ujBLYM} we get
\begin{equation}\label{eq-L-MLYM2}
\sum_{E\in\FF(F;P)}\frac{1}{\prod\limits_{i\in P} \binom{m_i}{|E\cap X_i|}}\leq L_P.
\end{equation}
From this we can write for any fixed sequence $f_i$ $(i\notin P)$:
$$
\sum_{\genfrac{}{}{0pt}{}{F: F\subseteq X\setminus \bigcup_{i\in P}X_i} {|F\cap X_i|=f_i,
i\notin P}} \sum_{E\in\FF(F;P)} \frac{1}{{\prod\limits_{i\in P} \binom{m_i}{|E\cap X_i|}} 
\prod\limits_{i\notin P}\binom{m_i}{f_i}}\leq L_P.
$$
Finally, summing up the previous inequality for $f_i=0,1,\ldots,m_i$, for all $i\notin P$,
we obtain the theorem. \hfill\qed\\

To prove Theorem~\ref{lemma:k=M}, we first need the following definitions: Let $\FF\Subset X$ be an non-empty $M$-part
multi-family, and
let $j\in[M]$. We define $\high(\FF)$ and $\low(\FF)$ as the largest and smallest levels in $X_j$ that the trace
$\FF_{X_j}$ in $X_j$ intersects. With other words,
\begin{eqnarray*}
\high(\FF)&=&\max\left\{q\in[n_j]^{\star}: \FF_{X_j}\cap\binom{X_j}{q}\ne\emptyset\right\},\\
\low(\FF)&=&\min\left\{q\in[n_j]^{\star}: \FF_{X_j}\cap\binom{X_j}{q}\ne\emptyset\right\}.
\end{eqnarray*}
First, we will need the following:
\begin{lemma}\label{lemma:shadow}
Let $M>1$, $j\in[M-1]$ and let $\FF^{\prime}$ be an $M$-dimensional $M$-part Sperner family with
$\high(\FF^{\prime})>\low(\FF^{\prime})$ that satisfies {\rm (\ref{k=M:equality})}, and let $E_0\in\FF^{\prime}_{X_M}$ be fixed.
Then there is an $M$-dimensional $M$-part Sperner family $\FF$ 
that also satisfies {\rm (\ref{k=M:equality})} such that for all $i\in[M]\setminus\{j\}$ we have $\FF_{X_i}\subseteq\FF^{\prime}_{X_i}$,
$E_0\in\FF_{X_M}$ and $\high(\FF)-\low(\FF)=\high(\FF^{\prime})-\low(\FF^{\prime})-1$.
\end{lemma}

\begin{proof}
Let $t=\high(\FF^{\prime})$ and
$\BB=\{B_1,\ldots,B_s\}= \FF^{\prime}_{X_j}\cap \binom {X_j}{t}$.
For $i\in[s]$ let
 $\EE_i=\{E\in \FF^{\prime}: E\cap X_j= B_i\}$ and $\EE=\cup_{i=1}^s\EE_i$.
 Given $A\subseteq X_j$, we define
 $$w(A)=\sum_{E\in\FF^{\prime}: E\cap X_j=A}\frac{1}{\prod\limits_{i:i\ne j}\binom {m_i}{|E\cap X_j|}}.$$
We  also assume (w.l.o.g.) that $w(B_1)\geq\ldots\geq w(B_s)$; we will use $w_i:=w(B_i)$.
Using this notation we can rewrite (\ref{k=M:equality}) as
 \begin{equation*}
 \sum_{A\subseteq X_j}\frac{w(A)}{\binom{m_j}{|A|}}=1,
 \end{equation*}
 or equivalently
 \begin{equation}\label{k=M:equality2}
 \sum_{i=1}^s\frac{w_i}{\binom {m_j}{t}}+\delta=1;~~\delta:=\sum_{A\in \FF^{\prime}_{X_j}\setminus\BB}\frac{w(A)}{\binom{m_j}{|A|}}.
 \end{equation}
Recall the following well-known fact (see e.g. \cite{engelbook})
 that for every $t\in[n]$ and a  subset $\AA\subseteq \binom{[n]}{t}$
we have
\begin{equation}\label{shadow}
\frac{|\partial(\AA)|}{\binom{n}{t-1}}\geq\frac{|\AA|}{\binom{n}{t}},
\end{equation}
 where $\partial(\AA)$, called the {\sl lower shadow} of $\AA$, is defined as
$\partial(\AA)=\{E\in \binom{[n]}{t-1}: E\subsetneq F~ \text{for some}~ F\in \AA\}.$
Moreover, equality in (\ref{shadow}) holds if and only if $\AA=\binom{[n]}{t}$.

Similar inequality holds for the {\sl upper shadow} $\upshad(\AA)$ of $\AA$ defined as 
$\upshad(\AA)=\{E\in \binom{[n]}{t+1}: E\supsetneq F~ \text{for some}~ F\in \AA\}$,
that is $|\upshad(\AA)|/\binom {n}{t+1}\geq |\AA|/\binom {n}{t}$ (with equality if and only if $\AA=\binom{[n]}{t}$).

Let us denote $ \Bb_i=\{B_1,\ldots,B_i\};~ i=1,\ldots,s$~ (thus $\Bb_i\subsetneq \Bb_{i+1}$ and $\Bb_s=\BB$).
We define then the following partition of $\partial(\BB)=\Bb^{\prime}_1\cup\ldots\cup \Bb^{\prime}_s$:
 $$
 \Bb^{\prime}_1=\partial(\Bb_1),  ~\Bb^{\prime}_i=\partial(\Bb_i)\setminus\partial(\Bb_{i-1});~ i=2,\ldots,s.
$$

Then, in view of (\ref{shadow}),  we have
\begin{equation}\label{shadow-match}
\frac{\sum_{\ell=1}^i|\Bb^{\prime}_{\ell}|}{\binom{m_j}{t-1}}=\frac{|\partial(\Bb_i)|}{\binom{m_j}{t-1}}\geq\frac{|\Bb_i|}{\binom{m_j}{t}}=\frac{i}{\binom{m_j}{t}};~ i=1,\ldots,s,
\end{equation}
with strict inequality if $s<\binom{m_j}{t}$.

Recall that  $\AA$ is an $M$-dimensional $M$-part Sperner family with parameter $L_{[M]}=1$ precisely when for
all $A,B\in\AA$ with $A\ne B$.
there is an $\ell\in[M]$ such that $A\cap X_{\ell}$ and $B\cap X_{\ell}$ are incomparable by the subset relation.

For ease of description, let us represent each family $\EE_i$, defined above, by the direct product $\EE_i=\{B_i\}\times\HH_i$,
 where $\HH_i=\FF^{\prime}(B_i;[M]\setminus\{j\})$ 
 is an $(M-1)$-dimensional $(M-1)$-part Sperner family in the partition set $\biguplus_{i\in[M]\setminus\{j\}} X_i$.

We now construct a new family $\EE^{\star}$ from $\EE$ as follows. We replace each $\EE_i$ by $\EE^{\star}_i:=\Bb^{\prime}_i\times \HH_i;~ i=1,\ldots,s$ and
define $\EE^{\star}=\cup_{i=1}^s\EE^{\star}_i$. Observe now that  
for each $A^{\star}\in\EE^{\star}_i$ there is an $A\in\EE_i$ such that $A^{\star}\cap X_{\ell}=A\cap X_{\ell}$ for all $\ell\in[M]\setminus\{j\}$ and
$A^{\star}\cap X_j\subsetneq A\cap X_j$. This implies that
 $\EE^{\star}\cap \FF^{\prime}=\emptyset$, since $\FF^{\prime}$ is an $M$-dimensional $M$-part Sperner family with parameter $L_{[M]}=1$.
  
 Moreover, it is not hard to see that $\FF^{\star}:=(\FF^{\prime}\setminus \EE)\cup\EE^{\star}$
 is an $M$-dimensional $M$-part Sperner family with parameter $L^{\star}_{[M]}=1$. If we have that $A,B$ are different elements of
 $\FF^{\prime}\setminus\EE$, then the required property follows from the fact that $A,B$ are both elements of $\FF^{\prime}$.
 If $A^{\star},B^{\star}$ are different elements of $\EE^{\star}$, then either $A^{\star}\cap X_j$ and $B^{\star}\cap X_j$ are
 both incomparable, or $A^{\star},B^{\star}\in\EE_i$ for some $i$, in which case the corresponding sets 
 $A,B\in\EE_i\subseteq \FF^{\prime}$ agree
 with $A^{\star},B^{\star}$ on $X\setminus X_j$ and $A^{\star}\cap X_j=B^{\star}\cap X_j=B_i$, from which the required property follows. Finally, take $A^{\star}\in \EE^{\star}_i$ for some $i$ and $B\in\FF^{\prime}\setminus\EE$, and let 
 $A\in\EE_i\subseteq\FF^{\prime}$ 
 be the corresponding set.
 If $A\cap X_j$ and $B\cap X_j$ are comparable, then from the fact that $t$ was the largest level of $\FF^{\prime}_{X_j}$ we get
 that $B\cap X_j\subseteq A^{\star}\cap X_j\subsetneq B_i=A\cap X_j$; and from the fact that $A,B$ are both elements
 of $\FF^{\prime}$ and $A\setminus X_j=A^{\star}\setminus X_j$ the required property follows. 
 
 Therefore
 $\FF^{\star}$ is an $M$-dimensional $M$-part Sperner family with parameter $1$. 
Thus, for $\FF^{\star}$ the following inequality must hold:
\begin{equation}\label{new-BLYM}
\sum_{E\in\FF^{\star}}\frac{1}{\prod\limits_{i=1}^M \binom{m_i}{|E\cap X_i|}}
=\sum_{i=1}^s\frac{|\Bb^{\prime}_i|\cdot w_i}{\binom{m_j}{t-1}}+\delta\leq 1.
\end{equation}
On the other hand,  (\ref{shadow-match}) together with $w_1\geq\ldots\geq w_s\ge 0=:w_{s+1}$ implies that
\begin{equation}\label{weight-shadow}
\sum\limits_{\ell=1}^s\frac{|\Bb^{\prime}_{\ell}|\cdot w_{\ell}}{\binom{m_j}{t-1}}
=
\sum_{i=1}^{s}\sum_{\ell=1}^{i}\frac{|\Bb^{\prime}_{\ell}|\cdot(w_{i}-w_{i+1})}{\binom{m_j}{t-1}}
\geq\sum\limits_{i=1}^{s}\frac{i\cdot(w_{i}-w_{i+1})}{\binom{m_j}{t}}
=\sum_{i=1}^s\frac{w_i}{\binom{m_j}{t}}.
\end{equation}

In fact, the latter means that $\BB=\binom{X_j}{t}$, otherwise we have strict inequality in (\ref{weight-shadow})
 a contradiction with (\ref{new-BLYM}), in view of (\ref{k=M:equality2}).
Thus,  for the new family $\FF^{\star}$ we have
$$
\sum_{E\in\FF^{\star}}\frac{1}{\prod\limits_{i=1}^M \binom{m_i}{|E\cap X_i|}}= 1.
$$

Moreover, $\high(\FF^{\star})=\high(\FF^{\prime})-1$ and $\low(\FF^{\star})=\low(\FF^{\prime})$, so
$\high(\FF^{\star})-\low(\FF^{\star})=\high(\FF^{\prime})-\low(\FF^{\prime})-1$.
In addition, for all $\ell\in[M]\setminus\{j\}$ we have $\EE^{\star}_{X_{\ell}}\subseteq\EE_{X_{\ell}}$, therefore
$\FF^{\star}_{X_{\ell}}\subseteq\FF^{\prime}_{X_{\ell}}$. Therefore, if $E_0\in(\FF^{\prime}\setminus\EE)_{X_M}$, i.e.
the trace of $\FF^{\prime}\setminus\EE$ in $X_M$ contains $E_0$, then setting $\FF:=\FF^{\star}$ will give the required family.

If $E_0\notin(\FF^{\prime}\setminus\EE)_{X_M}$, then, since 
$\FF^{\prime}_{X_M}\setminus\EE_{X_M}\subseteq(\FF^{\prime}\setminus\EE)_{X_M}$ and $E_0\in\FF^{\prime}_{X_M}$ we
must have that $E_0\in\EE_{X_M}$.
Similar to the described "pushing down" transformation in $\FF^{\prime}$ we can apply "pushing up" transformation  with respect to the smallest level $\DD$ in $\FF^{\prime}_{X_j}$, replacing it by its upper shadow $\upshad(\DD)$ to obtain
the new family $\FF$.  
Since $\DD\ne\BB$, we now have $\EE\subseteq\FF$, therefore $E_0\in\FF_{X_M}$.
All other required conditions
follow as before.
\end{proof}

\noindent{\sl Proof to Theorem~\ref{lemma:k=M}}: 
Let $\FF^{\prime}$ be an $M$-dimensional $M$-part Sperner family with parameter $1$ satisfying (\ref{k=M:equality}).
Without loss of generality assume, contrary to the statement of the theorem, that the trace $\FF^{\prime}_{X_M}$ 
of $\FF^{\prime}$ in $X_M$ 
contains an incomplete level, i.e. there is a $y_M\in[n_M]^{\star}$ such that for 
$\GG=\FF^{\prime}\cap\binom{X_M}{y_M}$ we have that $\emptyset\subsetneq\GG\subsetneq\binom{X_M}{y_M}$.  Fix an 
$E_0\in\GG$.

Let $\FF^{(0)}:=\FF^{\prime}$. We will define a sequence $\FF^{(1)},\ldots,\FF^{(M-1)}$ of $M$-dimensional $M$-part Sperner families such that for each $\ell\in[M-1]$ the following hold:
\begin{enumerate}[(i)]
\item Equality
(\ref{k=M:equality}) holds for $\FF^{(\ell)}$, with other words
\begin{equation}\label{eq:newer-family}
\sum_{E\in\FF^{(\ell)}}\frac{1}{\prod\limits_{i=1}^M \binom{m_i}{|E\cap X_i|}}=1.
\end{equation}
\item There is a $y_{\ell}\in[n_{\ell}]^{\star}$ such that $\FF^{(\ell)}_{X_{\ell}}\subseteq\binom{X_{\ell}}{y_{\ell}}$, with
other words the trace of $\FF^{(\ell)}$ in $X_{\ell}$ consist of a single (not necessarily full) level.
\item For each $i\in[M]\setminus\{\ell\}$, $\FF^{(\ell)}_{X_i}\subseteq\FF^{(\ell-1)}_{X_i}$.
\item $E_0\in\FF^{(\ell)}_{X_M}$.
\end{enumerate}
Once this sequence is defined, it follows that for all $j\in[M-1]$ we have that $\FF^{(M-1)}_{X_j}\subseteq\binom{X_j}{y_j}$,
also $E_0\in\left(\FF^{(M-1)}_{X_M}\cap\binom{X_{M}}{y_M}\right)\subseteq\GG\subsetneq\binom{X_{M}}{y_M}$, therefore
the trace of $\FF^{(M-1)}$ in $X_M$ contains an incomplete level.

 Also,
 for all $F\in X\setminus X_M$ we must have that
 $\FF^{(M-1)}(F;\{M\})$ is a $1$-dimensional $1$-part Sperner family with parameter $1$, 
 therefore it satisfies (\ref{eq:multi-BLYM}) with the parameter set to $1$.
In view of these facts, using (\ref{eq:newer-family}) for $\ell=M-1$ we get that
\begin{eqnarray*}
1&=&\sum_{E\in\FF^{(M-1)}}\frac{1}{\prod\limits_{i=1}^{M}\binom{m_i}{|E\cap X_i|}}\\
&=&
\frac{1}{\prod\limits_{i=1}^{M-1}\binom{m_i}{t_i}}\sum_{F\subseteq X\setminus X_M}
\left(\sum_{E\in\FF^{(M-1)}(F;\{M\})}\frac{1}{\binom{m_M}{|E\cap X_M|}}\right)\\
&\le&\frac{1}{\prod\limits_{i=1}^{M-1}\binom{m_i}{t_i}}\sum_{F\subseteq X\setminus X_M} 1
=1.
\end{eqnarray*}
This implies that for all $F\subseteq X\setminus X_M$, (\ref{eq:multi-BLYM}) holds with equality for $\FF^{(M-1)}(F,\{M\})$, 
so by Lemma~\ref{multiset}
we get that $\FF^{(M-1)}(F,\{M\})$ is a full level. Since
$$\FF^{(M-1)}_{X_M}=\bigcup_{F\subseteq X\setminus X_M}\FF^{(M-1)}(F,\{M\}),
$$
this implies that $\FF^{(M-1)}_{X_M}$ must consist of full levels only, a contradiction.

Note that $\FF^{(0)}$ is defined, it satisfies (\ref{k=M:equality}), and it does not need to satisfy any other conditions.
All that remains to show is that $\FF^{(\ell)}$ can be defined for each $\ell\in[M-1]$ such that it satisfies
the conditions (i)--(iv). 

To this end, assume that $j\in[M-1]$ and $\FF^{(j-1)}$ is already given satisfying all required conditions.
Let $Q=\high(\FF^{(j-1)})-\low(\FF^{(j-1)})$. If $Q=0$,
then $\FF^{(j-1)}_{X_j}$ consists of a single, not necessarily full, level, and we set $\FF^{(j)}=\FF^{(j-1)}$; (i)--(iv) are
clearly satisfied.

If $Q>0$, then let $\KK^{(0)}=\FF^{(j-1)}$. By Lemma~\ref{lemma:shadow} we can define a sequence 
$\KK^{(1)},\ldots,\KK^{(Q)}$ of $M$-dimensional $M$-part Sperner families with parameter $1$ such that for all
$\ell\in[Q]$ the following hold:
\begin{enumerate}[(a)]
\item $\KK^{(\ell)}$ satisfies {\rm (\ref{k=M:equality})}.
\item For all $i\in[M]\setminus\{j\}$ we have $\KK^{(\ell)}_{X_i}\subseteq\KK^{(\ell-1)}_{X_i}$.
\item $E_0\in\KK^{(\ell)}_{X_M}$.
\item $\high(\KK^{(\ell})-\low(\KK^{(\ell)})=\high(\KK^{(\ell-1)})-\low(\KK^{(\ell-1)})-1$.
\end{enumerate}
It follows that $\high(\KK^{(Q)})=\low(\KK^{(Q)})$ and we set $\FF^{(j)}=\KK^{(Q)}$; (i)--(iv) are clearly satisfied.
\hfill\qed
\bigskip

It only remains to prove Theorem~\ref{ujBLYM=}. We will start with a series of lemmata. The first lemma states for 
multi-families what Theorem 6.2 in \cite{tour} stated for simple families:
\begin{lemma}\label{1BLYM=} Let $1\le M$ and $\FF$ be a $1$-dimensional $M$-part Sperner multi-family with parameters 
$L_{\{i\}}$ for  $i\in[M]$ satisfying
{\rm(\ref{allBLYM})} with equalities, i.e. 
\begin{equation}\label{all1BLYM}
\forall i\in[M]\,\,\,\,\,\,\,\,\,
\sum_{(i_1,\ldots,i_M)\in\pi_M}\frac{p_{i_1,\ldots,i_M}}{\prod\limits_{j=1}^M \binom{m_j}{i_j}}=\frac{L_{\{i\}}}{n_i}\prod_{j=1}^M n_j.
\end{equation} 
Then $\FF$ is homogeneous. 
\end{lemma}
\begin{proof} For $M=1$ the statement is proved in Lemma~\ref{multiset}. 
Let $M\ge 2$ and take an arbitrary $F\in\FF$. We set $F_i=F\cap X_i$ and $G_i=F\setminus F_i$.
By Lemma~\ref{lemma:FF(F)} that for each $j\in[M]$,
$\FF(G_j;\{j\})$ is a ($1$-dimensional $1$-part) Sperner multi-family with parameter $L_{\{j\}}$.
From the proof of Theorem~\ref{ujBLYM} and (\ref{all1BLYM}) we get
that equality must hold in (\ref{eq-L-MLYM2}), i.e.
\begin{equation*}
\sum_{E\in\FF(G_j;\{j\})}\frac{1}{\binom{m_i}{|E\cap X_i|}}=L_{\{j\}},
\end{equation*}
which by Lemma~\ref{multiset} implies that $\FF(G_j;\{j\})$ is homogeneous. In particular this means that
for all $A\in\FF$ for all $j\in[M]$ if $B$ is a set such that
$|A\cap X_j|=|B\cap X_j|$ and
for all $i\in[M]\setminus\{j\}$ we have $A\cap X_i=B\cap X_i$, then
the $\#[A,\FF]=\#[B,\FF]$. 
If $A,B$ are sets with the same profile vector, we define the sequence 
$A=Y_0,Y_1,\ldots,Y_M=B$ by $Y_{i}=(Y_{i-1}\setminus X_{i})\uplus (B\cap X_{i})$ for all $i\in[M]$.
It follows that $\#[Y_{i-1},\FF]=\#[Y_{i},\FF]$, and so $\#[A,\FF]=\#[B,\FF]$. Thus $\FF$ is homogeneous.
\end{proof}
\begin{lemma}\label{kk+1BLYM=} Let $1\le k$ and let $\FF$ be a $k$-dimensional 
$(k+1)$-part Sperner multi-family with parameters 
$L_{[k+1]\setminus\{i\}}$ for $i\in[k+1]$ satisfying
{\rm(\ref{allBLYM})} with equality, i.e. 
\begin{equation}\label{allkk+1BLYM}
\forall i\in[k+1]\,\,\,\,\,\,\,\,\,
\sum_{(i_1,\ldots,i_{k+1})\in\pi_{k+1}}\frac{p_{i_1,\ldots,i_{k+1}}}{\prod\limits_{j=1}^{k+1} \binom{m_j}{i_j}}
=L_{[k+1]\setminus\{i\}}n_i.
\end{equation} 
Then $\FF$ is homogeneous. 
\end{lemma}
\begin{proof} The proof is induction on $k$.
For $k=1$, it is proved in Lemma~\ref{1BLYM=}.
By Lemma~\ref{lemma:FF(F)} we have that for each $j\in[M]$
and each $F\subseteq X_j$, $\FF(F;[k+1]\setminus\{j\})$ is a $k$-dimensional $k$-part Sperner multi-family with parameter
$L_{[k+1]\setminus\{j\}}$. From
the proof of Theorem~\ref{ujBLYM} and (\ref{allkk+1BLYM}) we get
that equality must hold in (\ref{eq-L-MLYM2}), i.e.
\begin{equation}\label{eq-L-k+1LYM}
\sum_{E\in\FF(F;[k+1]\setminus\{j\})}\frac{1}{\prod\limits_{i:i\ne j}\binom{m_i}{|E\cap X_i|}}=L_{[k+1]\setminus\{j\}}.
\end{equation}
Fixing a maximal chain $F_0\subsetneq F_1\subsetneq\cdots\subsetneq F_{m_j}$ in $X_j$,
we get that $\FF^{\prime}=\biguplus_{q=0}^{m_j} \FF(F_q;[k+1]\setminus\{j\})$ is a
 $(k-1)$-dimensional $k$-part Sperner multi-family
with parameters $L^{\prime}_{[k+1]\setminus\{j,\ell\}}:=L_{[k+1]\setminus\{\ell\}}: ~\ell\in[k+1]\setminus\{j\}$, 
moreover, using (\ref{eq-L-k+1LYM}) for
each $F=F_q$ we get that
\begin{equation*}
\sum_{E\in\FF^{\prime}}\frac{1}{\prod\limits_{i:i\ne j}\binom{m_i}{|E\cap X_i|}}
=\sum_{q=0}^{m_j}
\left(\sum_{E\in\FF(F_q;[k+1]\setminus\{j\})}\frac{1}{\prod\limits_{i:i\ne j}\binom{m_i}{|E\cap X_i|}}\right)=n_jL_{[k+1]\setminus\{j\}}.
\end{equation*}
By (\ref{allkk+1BLYM}) we have that
$L_{[k+1]\setminus\{j\}}n_j= L_{[k+1]\setminus\{\ell\}}n_{\ell}=L^{\prime}_{[k+1]\setminus\{j,\ell\}}n_{\ell}$, therefore 
 $\FF^{\prime}$ is homogeneous by the induction hypothesis.
In particular this means that
for all $A\in\FF$ for all $j\in[M]$ if $B$ is a set such that
$|A\cap X_j|=|B\cap X_j|$ and
for all $i\in[M]\setminus\{j\}$ we have $A\cap X_i=B\cap X_i$, then
the $\#[A,\FF]=\#[B,\FF]$. This implies, as in the proof of Lemma~\ref{1BLYM=}, 
that $\FF$ is homogeneous.
\end{proof}
\begin{lemma}\label{kMBLYM=} Let $2\le k\le M-1$ and let $\FF$ be a $k$-dimensional 
$M$-part Sperner multi-family with parameters 
$L_{P}$ for $P\in\binom{[M]}{k}$ satisfying 
{\rm(\ref{allBLYM})} with equalities. Then $\FF$ is homogeneous. 
\end{lemma}
\begin{proof} The proof is essentially the same as the proof of Lemma~\ref{kk+1BLYM=}.
If $M=k+1$, we are done by Lemma~\ref{kk+1BLYM=}. If $M>k+1$, by Lemma~\ref{lemma:FF(F)}
we get that for each $D\in\binom{[M]}{k+1}$ and
$F\subseteq X\setminus \biguplus_{i\in D} X_i$, 
$\FF(F;D)$ is a $k$-dimensional $(k+1)$-part Sperner multi-family with parameters
$L_{P}:~P\in\binom{D}{k}$. Fix an $F\subseteq X\setminus \biguplus_{i\in D} X_i$, and set $\FF^{\prime}=\FF(F;D)$.
For any $j\in D$ and $G\subseteq X_j$ we have
that $\FF^{\prime}(G;D\setminus\{j\})=\FF(F\uplus G;D\setminus \{j\})$ and $\FF^{\prime}(G;D\setminus\{j\})$
is a $k$-dimensional $k$-part Sperner family with parameter $L_{D\setminus \{j\}}$. 
From
the proof of Theorem~\ref{ujBLYM} and the fact that in $\FF$ (\ref{allBLYM}) holds with equality we get
that equality must hold for all $j\in D$ and all $G\subseteq X_j$ for 
$\FF^{\prime}(G;D\setminus\{j\})$ in (\ref{eq-L-MLYM2}), i.e.
\begin{equation}\label{eq-DLYM}
\sum_{E\in\FF^{\prime}(G;D\setminus\{j\})}\frac{1}{\prod\limits_{\ell\in[M]\setminus(D\cup\{j\})}\binom{m_{\ell}}{|E\cap X_{\ell}|}}
=L_{D\setminus\{j\}}.
\end{equation}
Fixing a maximal chain $G_0\subsetneq G_1\subsetneq\cdots\subsetneq G_{m_j}$ in $X_j$ 
we get that $\FF^{\star}=\biguplus_{i=0}^{m_j} \FF^{\prime}(G_i;D\setminus\{j\})$ is a $(k-1)$-dimensional $k$-part Sperner multi-family
with parameters $L^{\star}_{P^{\star}}:=L_{P^{\star}\cup\{j\}}:~ P^{\star}\in\binom{D\setminus\{j\}}{k-1}$, moreover, using 
(\ref{eq-DLYM}) for
each $G_i$ we get that \begin{equation*}
\sum_{E\in\FF^{\star}}\frac{1}{\prod\limits_{\ell\in D\setminus\{j\}}\binom{m_{\ell}}{|E\cap X_{\ell}|}}
=\sum_{i=1}^{m_j}\left(
\sum_{E\in\FF(G_i;D\setminus\{j\})}\frac{1}{\prod\limits_{\ell\in D\setminus\{j\}}\binom{m_{\ell}}{|E\cap X_{\ell}|}}\right)=L_{D\setminus\{j\})}n_j.
\end{equation*}
Fix any $P^{\star}\in\binom{D\setminus\{j\}}{k-1}$. Then
$P^{\star}=D\setminus\{i,j\}$ for some $i\in D\setminus\{j\}$, and from the conditions of the theorem we get 
that 
$$L^{\star}_{P^{\star}}n_{i}=
L_{P^{\star}\cup\{j\}}n_i=L_{D\setminus \{i\})}n_i=
L_{D\setminus\{j\})}n_j,$$
therefore
 $\FF^{\star}$ is homogeneous by the induction hypothesis. This means that if $A\in\FF$ and $B$ is a set with the same profile vector as
 $A$, and $A\cap X_i=B\cap X_i$ for at least $M-k-1\geq 1$ values of $i$, then $\#[A,\FF]=\#[B,\FF]$.
As before, we get
that $\FF$ is homogeneous.
\end{proof}

\noindent{\sl Proof to Theorem~\ref{ujBLYM=}}:
Lemmata~\ref{1BLYM=}, \ref{kk+1BLYM=} and \ref{kMBLYM=} together proves part (i), and, as remarked earlier, part (ii) follows from the conditions.

\noindent
(iii): By part (i), equality in (\ref{allBLYM}) implies homogeneity, i.e. that for any $(i_1,\ldots,i_M)\in\pi_M$ there is a
positive integer $r_{i_1,\ldots,i_M}$ such that every set in $\FF$ that has profile vector $(i_1,\ldots,i_M)$ appears with multiplicity 
$r_{i_1,\ldots,i_M}$, and also equality in (\ref{eq-L-MLYM2}). Equality in (\ref{eq-L-MLYM2}) means
that for any chain product  $\CC:=\prod^M_{i=1}\CC_i$ where $\CC_i$ is a maximal chain in $X_i$, 
any given $P\in\binom {[M]}k$ and  any subset $F\subseteq X\setminus \cup_{i\in P} X_i$,
 each subproduct $\prod_{j\in P}\CC_j$  of maximal chains is covered exactly $L_{P}$ times by the elements of $\FF(F;P)$,
that is 
\begin{equation}
\label{eq-prop}
\Bigg\vert\Bmultiset{E^{\#[E,\FF(F;P)]}:E\in\prod_{j\in P}\CC_j}\Bigg\vert=L_{P}.
\end{equation}
For a given chain product $\CC=\prod_{i\in[M]}\CC_i$ of maximum-size chains $\CC_i$  in $X_i$, we define
\begin{equation*}
\FF[\CC]=\bmultiset{F^{\#[F,\FF]}: F\in\CC}.
\end{equation*}
Each  $F\in\FF[\CC]$ is uniquely determined from its profile vector $(f_1,\ldots,f_M)$. 
Let $\TT_{\CC}$ denote the multiset of all  profile vectors of the sets in $\FF[\CC]$, where
each profile vector appears with the multiplicity of its corresponding set in $\FF[\CC]$. Since
$\FF$ is homogeneous, $\TT_{\CC}$ does not depend on the choice of $\CC$.

We can describe now property (\ref{eq-prop}) of $\FF[\CC]$ in terms of its profile vectors as follows: 
for each subset $\{i_1<\ldots<i_{M-k}\}\in \binom {[M]}{M-k}$, and each $(M-k)$-tuple of coordinate values 
$(f_{i_1},\ldots,f_{i_{M-k}})\in\prod_{j=1}^{M-k} [n_{i_j}]^{\star}$ 
the set of vectors in $\TT_{\CC}$
where the $i_j$-th coordinate is $f_{i_j}$ for $j\in[M-k]$ has size $L_{M\setminus\{i_1,\ldots,i_{M-k}\}}$.
 Let $\TT$ denote the transversal corresponding to the homogeneous multi-family $\FF$.
Then clearly $\TT=\TT_{\CC}$ for every product of maximal chains $\CC=\prod^M_{i=1}\CC_i$.

We infer now that 
the $k$-dimensional multi-transversal $\TT$ is a simple MOA with symbol sets
$S_i=\{0,1,2\ldots,m_i\}$, of constraint
$M$, strength $M-k$, and index set $\mathbb{L}=\{L_P:P\in\binom{[M]}{k}\}$, with $\lambda({j_1},\ldots,{j_{M-k}})=
L_{[M]\setminus \{j_1,\ldots,j_{M-k}\}}$. This completes the proof of part (iii).

It is also clear that any MOA
 with the parameters described above is a $k$-dimensional
multi-transversal corresponding to a homogeneous $k$-dimensional $M$-part Sperner multi-family  $\FF$ with parameters
 $\{L_P: P\in\binom{[M]}{k}\}$   on a partitioned  $(m_1+\ldots+m_M$)-element underlying set, where the multiplicity of each element
 in $F\in\FF$ is the same as the multiplicity of its profile vector $(f_1,\ldots,f_M)$ in the multi-transversal,
 which  satisfies equality in (\ref{allBLYM}).
\hfill\qed

\section{Proofs to convex hull results}\label{hullproofs}

\noindent{\sl Proof to Lemma~\ref{th-blow}:} 
 will suffice to show that for every multi-family  $\HH\in \AAA$, there are non-negative coefficients $\lambda(I) $ for 
every $I\Subset\pi_M$  with $T(I)\in \mu(   \AAA(\LL))$, such that $\sum_I \lambda(I)=1$
and 
\begin{equation}\label{kell}
\sum_I \lambda(I)S(I)=\Pp(\HH).
\end{equation}
To this end, fix an $\HH\in \AAA$ and for all $H\subseteq X$ let $\HH_H=\multiset{H^{\#[H,\HH]}}$, with other words $\HH_H$ has $H$ with the same multiplicity as $\HH$, and it has no other elements.
Consider the sum
\begin{equation} \label{pthree}
\sum_{(\LL,H)} \frac{\Pp(\HH_H)}{\prod\limits_{j=1}^M (m_j!)}
\end{equation}
for all ordered pairs   $ (\LL,H) $,  where $\LL$ is a product-permutation, $H\subseteq X$, and $H$ is initial 
with respect to the
product-permutation $\LL$. We evaluate (\ref{pthree}) in two ways. The first way is:
\begin{eqnarray}  \nonumber
\sum_{(\LL,H)} \frac{\Pp(\HH_H)}{\prod\limits_{j=1}^M (m_j!)}&=&\sum_{\LL} \frac{1}{\prod\limits_{j=1}^M (m_j!)}    
\left(\sum_{H\subseteq X:\atop H \text{ is initial for } \LL}       \Pp(\HH_H)\right)\\
&=& \sum_{\LL} \frac{\Pp( \HH(\LL)) }{\prod\limits_{j=1}^M (m_j!)}.  \label{pfour}
\end{eqnarray}
Observe that $\Pp(\HH(\LL))\in \mu(\AAA(\LL))$, and therefore for every $\LL$ there is a unique $I$ such that $T(I)=\Pp(\HH(\LL))$. Collecting the identical
terms in the right side of  (\ref{pfour}),
\begin{equation} \label{pfive}
\sum_{\LL} \frac{\Pp(\HH(\LL))}{\prod\limits_{j=1}^M (m_j!)}=\sum_{T(I)\in \mu(\AAA(\LL))} \lambda(I) T(I),
\end{equation}
where $\lambda(I)$ is the proportion of the $\prod_{j=1}^M (m_j!)$ product-permutations such that $\Pp(\HH(\LL))=T(I)$,
thus $\sum_{T(I)\in \mu(\AAA(\LL))} \lambda(I)=1$.  
Consider a fixed set $H$ with profile vector $(i_1,i_2,\ldots,i_M)$. There are exactly $\prod_{j=1}^M(i_j!\cdot(m_j-i_j)!)$
product-chains to which $H$ is initial.  Using this, we also get: 
\begin{eqnarray}  \nonumber
\sum_{(\LL,H)} \frac{\Pp(\HH_H)}{\prod\limits_{j=1}^M(m_j!)}
&=&\sum_{H:H\subseteq X}\sum_{\LL:\atop H\text{\ is initial for\ }\LL} \frac{\Pp(\HH_H) }{\prod\limits_{j=1}^M(m_j!)}        \\
&=& \sum_{H:H\subseteq X}    \frac{ \prod\limits_{j=1}^M(i_j!\cdot(m_j-i_j)!) }{\prod\limits_{j=1}^M(m_j!)}\cdot\Pp(\HH_H)  \label{pseven}\\
&=&  \sum_{H:H\subseteq X} \frac{\Pp(\HH_H)   }{ \prod\limits_{j=1}^M\binom{m_j}{i_j}}
=\Biggl(\frac{p_{i_1,\ldots,i_M}(\HH)}{\binom{m_1}{i_1}\cdots  \binom{m_M}{i_M }}       \Biggl)_{(i_1,\ldots,i_M)\in\pi_M} 
\label{pnyolc}.
\end{eqnarray}
Combining (\ref{pthree}), (\ref{pfour}), (\ref{pfive}), (\ref{pseven}),  and (\ref{pnyolc}), we obtain 
$$  \Biggl(\frac{p_{i_1,\ldots,i_M}(\HH)}{\binom{m_1}{i_1   }\cdots  \binom{m_M}{i_M   }}       \Biggl)_{(i_1,\ldots,i_M)\in\pi_M} 
=     \sum_{T(I)\in \mu(\AAA(\LL))} \lambda(I) T(I),$$
which implies for all $(i_1,\ldots,i_M)\in\pi_M$ that
$$  p_{i_1,\ldots,i_M}(\HH)
=     \sum_{T(I)\in \mu(\AAA(\LL))}\left( \lambda(I) \left(\prod_{j=1}^M\binom{m_j}{i_j}\right)t_{i_1,\ldots,i_M}(I)\right).$$
This proves (\ref{kell}).
\hfill\qed
\bigskip

\noindent{\sl Proof to Theorem~\ref{hullappl}:} 
First observe that $\mu(\AAA(\LL))$ does not depend on $\LL$, so (\ref{eq-blow2}) holds. Next we have to show 
(\ref{eq-blow3}), i.e. we have to show that if 
$T(I) \in \mu(\AAA(\LL))$ for some $I\Subset\pi_M$ and all product-permutation $\LL$, then  $ S(I) \in
  \mu(\AAA)$. 
  
  Assume $I\Subset\pi_M$ and $T(I) \in \mu(\AAA(\LL))$ for all product-permutation $\LL$. 
  Then for each product-permutation $\LL$
  there is an $\HH_{\LL}\in\AAA$ such that
 $T(I)=\Pp(\HH_{\LL}(\LL))$.
 Since $\HH_{\LL}$, and therefore $\HH_{\LL}(\LL)$ as well, satisfies $\MMM_{\Gamma}$, we must have that
 $I$ satisfies $\MMM_{\Gamma}$. Let $\FF_{S(I)}$ be the homogeneous multi-family that realizes the profile matrix $S(I)$, then for
 all $(i_1,\ldots,i_M)\in\pi_M$ we have
 $\max\{\#[F,\FF_{S(I)}]:\forall j\,\,|F\cap X_j|=i_j\}=\#[(i_1,\ldots,i_M),I]$, consequently, $\FF_{S(I)}$ satisfies $\MMM_{\Gamma}$.
Thus $S(I)\notin \mu(\AAA)$ implies that the homogeneous multi-family 
$\FF_{S(I)}$ is not a $k$-dimensional $M$-part multi-family with parameters $L_P:~P\in\binom{[M]}{k}$. 
This means that there is a $P_0\in\binom{[M]}{k}$, sets $D_i$  for all $i\notin P_0$ and chains $\CC_j$  for all $j\in P_0$ such that
 \begin{equation}
\Bigg\vert
\Bmultiset{
F^{\#[F,\FF_{S(I)}]}:
\Big(F\cap\biguplus_{j\in P_0} X_j\Big)\in\prod_{j\in P_0} \CC_j, \forall i\in[M]\setminus P_0\,\,\, X_i\cap F=D_i
}
\Bigg\vert> L_{P_0}.
\label{eq-fail1}
\end{equation}
Take now a product-permutation $\LL_0$ in which all 
sets $D_i$ ($i\notin P_0$) and all elements of the chains $\CC_j$  ($j\in P_0$) are initial with respect to $\LL_0$. 
Since
$\Pp(\FF_{S(I)}(\LL_0))=T(I)$ we can rewrite (\ref{eq-fail1}) as
\begin{equation}\label{eq-fail2}
\sum_{(i_1,\ldots,i_M)\in\pi_M:\atop
\forall j\notin P_0\,\,i_j=|D_j|} t_{i_1,\ldots,i_M}(I)> L_{P_0}.
\end{equation}
 As $T(I)=\Pp(\HH_{\LL_0}(\LL_0))$,  (\ref{eq-fail2}) gives
\begin{equation}\label{eq-fail3}
\Bigg\vert
\Bmultiset{
F^{\#[F,\HH_{\LL_0}(\LL_0)]}:
\Big(F\cap\biguplus_{j\in P_0} X_j\Big)\in\prod_{j\in P_0} \CC_j, \forall i\in[M]\setminus P_0\,\,\, X_i\cap F=D_i
}
\Bigg\vert> L_{P_0}.
\end{equation}
However, from $\HH_{\LL_0}\in\AAA$ we get that $\HH_{\LL_0}$, and consequently $\HH_{\LL_0}(\LL_0)$ must
be $k$-dimensional $M$-part Sperner multi-families with parameters $L_P:P\in\binom{[M]}{k}$, contradicting (\ref{eq-fail3}). 
\hfill\qed
\bigskip

\noindent{\sl Proof to Lemma~\ref{lemma:LEM}:} Let $\AAA$ be family of $k$-dimensional $M$-part Sperner multi-families
that satisfy a $\Gamma$-multiplicity contraints $\MMM_{\Gamma}$, and let $I\Subset \pi_M$ be a
 $k$-dimensional multi-transversal 
with the same parameters $L_P$ satisfying the same $\Gamma$-multiplicity constraint $\MMM_{\Gamma}$.
Let $\LL$ be a fixed product-permutation, for each $(i_1,\ldots,i_M)\in I$ let $H_{(i_1,\ldots,i_M)}$ be 
the (unique) initial set with respect to $\LL$ with
profile vector $(i_1,\ldots,i_M)$ and let
$$\HH_{\LL}=\multiset{H_{(i_1,\ldots,i_M)}^{t_{i_1,\ldots,i_M}(I)}}.$$
It follows that $\HH_{\LL}(\LL)=\HH_{\LL}$, $\Pp(\HH_{\LL})=T(I)$, and from the properties of $I$ we have
that $\HH_{\LL}\in\AAA$. Therefore we get that $T(I)\in\mu(\AAA(\LL))$. 
By Theorem~\ref{th-blow2}, the vector $S(I)$ is present in the set on the right hand side of 
(\ref{eq-blow4}), whose extreme points agree with those of $\mu(\AAA)$, and by Theorem~\ref{hullappl}, $S(I)\in\mu(\AAA)$. 
All that remains to be shown is that
if $S(I)=\sum_{ T(I_u)\in \mu(\AAA(\LL))} \lambda(I_u)S(I_u)$ with $\lambda(I_u)\geq 0$ and 
$\sum_{ T(I_u)\in \mu(\AAA(\LL))} \lambda(I_u)=1$, then $I$ is among the 
$I_u$'s, and all others come with a zero coefficient. 
$S(I)=\sum_{ T(I_u)\in \mu(\AAA(\LL))} \lambda(I_u)S(I_u)$ 
means that for all $(i_1,\ldots,i_M)\in\pi_M$ we have
$$t_{i_1,\ldots,i_M}(I)\prod_{j=1}^M\binom{m_j}{i_j}=\sum_{ T(I_u)\in \mu(\AAA(\LL))} \lambda(I_u)t_{i_1,\ldots,i_M}(I_u)\prod_{j=1}^M\binom{m_j}{i_j},$$
which implies that  
$$T(I)=\sum_{ T(I_u)\in \mu(\AAA(\LL))} \lambda(I_u)T(I_u).$$ 
Let the ordering $\supp(I)=\{\vec{j}_1,\vec{j}_2,\ldots,\vec{j}_s\}$ show that $I$ has the LEM property.
Then for all $u$, $T_{\vec{j}_1}(I)\geq T_{\vec{j}_1}(I_u)$, and as the coefficients sum to 1, for all $u$,
$T_{\vec{j}_1}(I)= T_{\vec{j}_1}(I_u)$. This argument repeats to $\vec{j}_2,\ldots,\vec{j}_s$. Hence
for all $u$, 
$\supp(I)\subseteq \supp(I_u) $. If $\supp(I)$ is a proper subset of $\supp(I_u) $, then we must have $\lambda(I_u)=0$. Therefore for all the $I_u$ that have $\lambda(I_u)\ne 0$ we must have $\supp(I_u)=\supp(I)$, and consequently
$I_u=I$.
\hfill\qed

\section{Proofs for the results on transversals}
\label{MOAproofs}
We start with two lemmata.
\begin{lemma}\label{2n} Let  us be given $n_1,n_2$ positive integers, and set $\ell=\gcd(n_1,n_2)$, $m_i=\frac{n_i}{e\\}$
and $N=\lcm(n_1,n_2)=\frac{n_1n_2}{e\\}$.
For every $j\in[N]^{\star}$, there are exactly $\ell$ vectors
$(a_1,a_2)\in\pi_2$, such that
$
\fracpart{\frac{a_1}{n_1}+\frac{a_2}{n_2}}=\frac{j}{N}.
$
\end{lemma}
\begin{proof}
Since $m_1,m_2$ are relatively prime, for any integer $j\in[N]^{\star}$
we have integers $z_1,z_2$ such that
$z_1m_2+z_2m_1=j$, therefore $\frac{z_1}{n_1}+\frac{z_2}{n_2}=\frac{j}{N}$. Taking
$a_i\in[n_i]^{\star}$ such that $a_i\equiv z_i\mod n_i$ we obtain that the required  vectors $(a_1,a_2)$
exist for any $j$. It is also clear that for any $(a_1,a_2)\in\pi_2$ there is some
$j\in[N]^{\star}$ such that $\fracpart{\frac{a_1}{n_1}+\frac{a_2}{n_2}}=\frac{j}{N}$.

So we define for any $j\in[N]^{\star}$
\begin{eqnarray*}
\DD_j=\left\{(a_1,a_2)\in\pi_2:\fracpart{\frac{a_1}{n_1}+\frac{a_2}{n_2}}=\frac{j}{N}\right\}.
\end{eqnarray*}
Fix $j\in[N]^{\star}$ and $(x_1,x_2)\in\DD_j$. For any $(y_1,y_2)\in\pi_2$ we have that
$(y_1,y_2)\in \DD_j$ iff $\frac{y_1-x_1}{n_1}+\frac{y_2-x_2}{n_2}$ is an integer.

The $\DD_j$ are nonempty and partition $\pi_2$. If
for each $j,j^{\prime}\in[N]^{\star}$, there is an injection from $\DD_j$ to $\DD_{j^{\prime}}$, then
$|\DD_j|=|\DD_{j^{\prime}}|$, and consequently
$|\DD_j|=\frac{n_1n_2}{N}=\ell$, which proves our statement. So we will construct such an injection.

Let $j,j^{\prime}\in[N]^{\star}$.
Fix an $(a_1,a_2)\in\DD_j$ and a $(b_1,b_2)\in\DD_{j^{\prime}}$. We define
the map $\phi:\DD_j\rightarrow\pi_2$ by
$\phi(c_1,c_2)=(d_1,d_2)\in\pi_2$ iff $d_i\equiv c_i+(b_i-a_i) \mod n_i$.
Clearly, the map is a well-defined injection, moreover, $\phi(a_1,a_2)=(b_1,b_2)$.

Assume that $(d_1,d_2)\in\phi(\DD_j)$. Then $(d_1,d_2)=\phi(c_1,c_2)$ for some
$(c_1,c_2)\in\DD_j$, and $d_i-b_i\equiv c_i-a_i\mod n_i$.
Thus
$\left(\frac{d_1-b_1}{n_i}+\frac{d_2-b_2}{n_2}\right)-\left(\frac{c_1-a_1}{n_1}+\frac{c_2-a_2}{n_2}\right)$ is an integer. Since $(c_1,c_2)\in \DD_j$, this implies
$\frac{d_1-b_1}{n_i}+\frac{d_2-b_2}{n_2}$ is also an integer, with other
words $(d_1,d_2)\in \DD_{j^{\prime}}$. Therefore $\phi(\DD_j)\subseteq\DD_{j^{\prime}}$.
\end{proof}

\begin{lemma}\label{generaln} Let $n_1,n_2,\ldots,n_k$ be given,
$K=\prod_{j=1}^{k} n_j$, $N=\lcm(n_1,\ldots,n_k)$ and
$\ell=\frac{K}{N}$.
For each $j\in[N]^{\star}$ we have that there are exactly $\ell$ vectors
$(a_1,\ldots,a_k)\in\pi_k$ such that
$$
\fracpart{\sum_{i=1}^k\frac{a_i}{n_i}}=\frac{j}{N}.
$$
\end{lemma}
\begin{proof} We prove the statement
by induction on $k$. The statement is clearly true for $k=1$(when $N=n_1$ and $\ell=1$); and it was proved in Lemma~\ref{2n} for $k=2$. So assume that $k>2$ and we know the statement already for all
$1\le k^{\prime}\le k-1$.

It is clear that for any $(a_1,\ldots,a_k)\in\pi_k$ we have precisely one
$j\in[N]^{\star}$ such
that $\fracpart{\sum_{i=1}^k\frac{a_i}{n_i}}=\frac{j}{N}$.
Let $K_1=\prod_{j=1}^{k-1} n_j$, $N_1=\lcm(n_1,\dots,n_{k-1})$ and $\ell_1=\frac{K_1}{N_1}$,
and $\ell_2=\gcd(N_1,n_k)$.
Then $K=K_1n_k$, $N=\lcm(N_1,n_k)$ and $\ell=\frac{K_1 n_k}{\lcm(N_1,n_k)}
=\frac{K_1}{N_1}\cdot\frac{N_1n_k}{\lcm(N_1,n_k)}=\ell_1\ell_2
$.

Fix a $j\in[N]^{\star}$. Note that for integers $a_i$,
$\fracpart{\sum_{i=1}^{k-1}\frac{a_i}{n_i}}\in\{\frac{j^{\prime}}{n}:j^{\prime}\in[N_1]^{\star}\}$,
and for real numbers $c,d$ we have $\fracpart{\fracpart{c}+\fracpart{d}}=\fracpart{c+d}$.
By Lemma~\ref{2n}, there are precisely $\ell_2$ pairs $(b,a_k)\in[N_1]^{\star}\times [n_k]^{\star}$ such
that $\fracpart{\frac{b}{N_1}+\frac{a_k}{n_k}}=\frac{j}{N}$. By the induction hypothesis
for each $b\in[N_1]$ there are precisely $\ell_1$ values $(a_1,\ldots,a_{k-1})\in\pi_{k-1}$
such that $\fracpart{\sum_{j=1}^{k-1}\frac{a_j}{n_j}}=\fracpart{\frac{b}{N_1}}$.
Since $\ell_1\ell_2=\ell$, the statement follows.
\end{proof}
\noindent{\sl Proof  to Lemma~\ref{Engelprime}:}
By Lemma~\ref{generaln} the statement is equivalent with
\begin{equation*}
\Bigg\vert\left\{j\in[N]^{\star}: \fracpart{\alpha+\frac{j}{N}}\in[\beta,\beta+\mu)\}
\right\}\Bigg\vert
\in\left\{\lceil\mu N\rceil,\lfloor \mu N\rfloor\right\}
\end{equation*}
which follows from Lemma~\ref{1dim}. \hfill\qed

\noindent{\sl Proof to Theorem~\ref{fotetel}:} Assume that $\mu$ satisfies condition
(\ref{gencond}) and $0\le\beta\le 1-\mu$.
Fix $P\in\binom{[M]}{k}$ and for each
$j\notin P$ fix a $b_j\in [n_j]$. Then Condition (\ref{korlat})
follows from Lemma~\ref{Engelprime} using
$\alpha=\sum_{j\notin P}\frac{b_j}{n_j}$; thus
$\mathbb{C}(n_1,\ldots,n_M;\beta,\mu)$ is a $k$-dimensional transversal
with the given parameters $L_P$.

Assume now further that for $P_0\in\binom{[M]}{k}$ we have that
$\mu=\frac{L_{P_0}}{K_{P_0}}$ (as this is equivalent with
$\mu=\min_{P}\frac{L_P}{K_P}$). Then
we have that
\begin{eqnarray*}
L_{P_0}\ge d_{P_0}\lceil\mu N_{P_0}\rceil
=d_{P_0}\left\lceil \frac{L_{P_0}}{d_{P_0}}\right\rceil\ge L_{P_0},
\end{eqnarray*}
which implies that $\mu N_{P_0}$ is an integer, i.e. by
Lemma~\ref{Engelprime} our transversal is full.
\hfill\qed

\section{Acknowledgements}

This research started at the ``Search Methodologies II'' workshop at the Zentrum f\"ur interdisziplin\"are Forschung of
Universit\"at Bielefeld, where the last two authors met Professor Ahlswede for the last time. Special thanks go to Professor Charles
Colbourn for his encouragement to continue our investigation in this direction.

\bibliographystyle{plain}

\end{document}